\documentclass[aop,MSNbibl,citesort,dvips]{arximspdf}
\usepackage{graphicx}

\doi{10.1214/10-AOP627}
\volume{40}
\issue{2}
\pubyear{2012}
\firstpage{578}
\lastpage{610}

\makeatletter

\newtheorem{theorem}{Theorem}[section]
\newtheorem{cor}[theorem]{Corollary}
\newtheorem{lem}[theorem]{Lemma}
\newtheorem{ppn}[theorem]{Proposition}

\newproclaim{defn}[theorem]{Definition}
\newproclaim{ex}[theorem]{Example}
\newproclaim{rmk}[theorem]{Remark}

\newcommand{\cN}{\mathcal{N}}
\newcommand{\dcapf}{d^R}
\newcommand{\dcapb}{d^L}
\newcommand{\pathspace}{\Gamma}
\newcommand{\pathspaceR}{\Gamma^R}
\newcommand{\pathspaceL}{\Gamma^L}
\newcommand{\initpathsR}{\Gamma_{\mathrm{init}}^R}
\newcommand{\initpathsL}{\Gamma_{\mathrm{init}}^L}
\newcommand{\simplepathsR}{\Gamma_{\mathrm{sim}}^R}
\newcommand{\simplepathsL}{\Gamma_{\mathrm{sim}}^L}
\newcommand{\simplepaths}{\Gamma_{\mathrm{sim}}}
\newcommand{\tspathsR}{\Gamma_{\mathrm{t.s.}}^R}
\newcommand{\tspathsL}{\Gamma_{\mathrm{t.s.}}^L}
\newcommand{\tspaths}{\Gamma_{\mathrm{t.s.}}}

\newcommand{\al}{\alpha}
\newcommand{\gam}{\gamma}
\newcommand{\ep}{\varepsilon}
\newcommand{\de}{\delta}
\newcommand{\ka}{\kappa}
\newcommand{\vph}{\varphi}
\newcommand{\Om}{\Omega}
\newcommand{\si}{\sigma}
\newcommand{\thet}{\theta}
\newcommand{\C}{\mathbb{C}}
\newcommand{\D}{\mathbb{D}}
\newcommand{\HH}{\mathbb{H}}
\newcommand{\N}{\mathbb{N}}
\newcommand{\PP}{\mathbb{P}}
\newcommand{\Q}{\mathbb{Q}}
\newcommand{\R}{\mathbb{R}}
\newcommand{\cH}{\mathcal{H}}
\newcommand{\incto}{\uparrow}
\newcommand{\imag}{\operatorname{Im}}
\newcommand{\real}{\operatorname{Re}}
\newcommand{\f}{\frac}
\newcommand{\defeq}{:=}
\newcommand{\defeqr}{=:}
\newcommand{\pd}{\partial}

\newcommand{\sle}{\mathrm{SLE}}
\newcommand{\cdist}{d_*}
\newcommand{\cdisthaus}{\cdist^\cH}

\newcommand{\capacity}{\operatorname{cap}}
\newcommand{\caraconv}{\stackrel{\mathit{Cara}}{\longrightarrow}}
\newcommand{\dstrong}{d_{\mathcal U}}

\makeatother

\begin{document}
\begin{frontmatter}

\title{Strong path convergence from Loewner driving function convergence}
\runtitle{Path convergence from Loewner driving convergence}

\begin{aug}
\author[A]{\fnms{Scott} \snm{Sheffield}\corref{}\thanksref{t1}\ead[label=e1]{sheffield@math.mit.edu}} and
\author[B]{\fnms{Nike} \snm{Sun}\thanksref{t2}\ead[label=e2]{nikesun@stanford.edu}}
\runauthor{S. Sheffield and N. Sun}
\affiliation{Massachusetts Institute of Technology and Stanford University}
\address[A]{Department of Mathematics\\
Massachusetts Institute of Technology\\
Cambridge, Massachusetts 02139\\
USA\\
\printead{e1}}
\address[B]{Department of Statistics\\
Stanford University\\
Stanford, California 94305\\
USA\\
\printead{e2}}
\end{aug}

\thankstext{t1}{Supported in part by NSF Grants DMS-06-45585 and OISE-0730136.}
\thankstext{t2}{Supported in part by NSF Grant DMS-08-06211 and a
DMS--VIGRE grant to Stanford Statistics Department.}

\received{\smonth{4} \syear{2010}}
\revised{\smonth{10} \syear{2010}}

%
\begin{abstract}
We show that, under mild assumptions on the limiting curve, a sequence
of simple chordal planar curves converges uniformly whenever certain
Loewner driving functions converge. We extend this result to random
curves. The random version applies in particular to random lattice
paths that have chordal $\sle_\ka$ as a scaling limit, with $\ka<8$
(nonspace-filling).

Existing $\sle_\ka$ convergence proofs often begin by showing that the
Loewner driving functions of these paths (viewed from $\infty$)
converge to Brownian motion. Unfortunately, this is not sufficient, and
additional arguments are required to complete the proofs. We show that
driving function convergence \textit{is} sufficient if it can be
established for both parametrization directions and a generic
observation point.
\end{abstract}

%
\begin{keyword}[class=AMS]
\kwd[Primary ]{60J67}
\kwd[; secondary ]{30C35}
\kwd{31A15}.
\end{keyword}
\begin{keyword}
\kwd{Schramm--Loewner evolutions}
\kwd{Loewner driving convergence}
\kwd{capacity}.
\end{keyword}

\end{frontmatter}

\section{Introduction}
\label{sec:intro}

The Loewner differential equation, first described by Loewner in 1923,
relates a planar self-avoiding curve to a real-valued continuous
function (the ``Loewner driving function'') via conformal mappings. It
was discovered by Schramm in \cite{schrammlerw} that if one takes
the driving function to be $\sqrt{\ka} W_t$ for $W$ a standard Brownian
motion, then the resulting random curves---called the Schramm--Loewner
evolution with parameter $\ka$ and denoted $\sle_\ka$---are conformally
invariant in law and satisfy a certain Markovian property (the ``domain
Markov property''). They are furthermore the only curves with these
properties, making them the ``universal'' candidate for the scaling
limit of many discrete planar models in statistical physics. Indeed,
since their introduction \cite{schrammlerw}, a number of discrete
random paths have been shown to converge to $\sle_\ka$ in the scaling
limit: in particular, loop-erased random walks and uniform spanning
tree boundaries ($\sle_2$ and $\sle_8$) \cite{lswlerw}, Gaussian free
field level lines and the harmonic explorer ($\sle_4$)
\cite{schrammsheffieldexplorer,schrammsheffieldgff}, percolation
cluster boundaries ($\sle_6$) \cite{smirnovperc,smirnovperclong,camianewman}
and Ising spin interfaces and FK cluster boundaries
($\sle_3$ and $\sle_{16/3}$)
\cite{smirnovtowards,smirnovrc,smirnovrcii,kemppainensmirnovrciii,chelkaksmirnovisinguniversality,chelkaksmirnovising}.

In each of the cases where convergence has been proved, a strong form
of convergence has been obtained: when the random lattice paths are
conformally mapped to continuous random paths on a fixed domain, one
obtains, as the mesh size tends to zero, convergence in law with
respect to the uniform or supremum-norm metric (modulo
reparametrization of the curves), which we denote by $\dstrong$ (see
Section \ref{subsec:curve_metrics}).\setcounter{footnote}{2}\footnote{The metric $\dstrong$ is
also sometimes called the ``Fr\'echet distance'' (see \cite{altgodau}
for background). For consistency, we will use only the term ``uniform
metric'' here.} For random variables on a separable metric space, there
are several equivalent ways to define convergence in law (also referred
to as convergence in distribution or weak convergence): in our setting,
a natural formulation is via the Skorohod--Dudley theorem~\cite{dudleyskorohod},
which states that random variables converge in law
if and only if they can be defined on a joint probability space in
which they converge almost surely. When speaking of random curves, we
will sometimes use the phrase ``uniform convergence'' as a shorthand
for ``convergence in law with respect to the uniform metric.''

Most existing $\sle$ convergence proofs have shown a weaker form of
convergence first, that of convergence of the Loewner driving function,
and have then used additional estimates from the discrete model to
deduce uniform convergence
\cite{lswlerw,schrammsheffieldexplorer,schrammsheffieldgff}. (The
arguments in \cite{smirnovperc,smirnovperclong,camianewman} contend
with these issues in a slightly different way (see also \cite{sun} for
a survey).) The goal of this article is to provide a more general
criterion for deducing uniform convergence which is less dependent on
specific features of the model at hand.

Specifically, we show that Loewner driving function convergence
actually \textit{implies} uniform convergence provided it can be
established for both parametrization directions and with respect to a
generic target:
\begin{theorem}
Let $D$ be a smooth bounded simply connected planar domain with marked
boundary points $a$ and $b$ (distinct), and let $(\gam^j)$ be a
sequence of random simple paths in $D$ traveling from $a$ to $b$. For
each $x \in D$, let $\psi_x$ be a conformal map from $D$ to the unit
disc with $\psi_x(x) = 0$. Let $d_{x}^{R}$ be the metric on paths
avoiding $x$ defined by
\[
d_{x}^{R}(\gam_1, \gam_2) \defeq|T_1 - T_2| +
\|W_{x,t\wedge T_1}(\gam_1) - W_{x,t\wedge T_2}(\gam_2)\|
_\infty,
\]
where $W_{x,\cdot}(\gam_i)$ is the radial Loewner driving function for
$\psi_x \circ\gam_i$, and $[0,T_i]$ is the (necessarily finite)
interval on which this function is defined (see Section~\ref
{subsec:curve_metrics}). Suppose that for all $x$ in a countable dense
subset of $D$, the $\gam^j$ and their time reversals $\gam^{j-}$
converge in law with respect to $d_{x}^{R}$ to chordal $\sle_\ka$
from $a$ to $b$ and from $b$ to $a$, respectively. Then the $\gam^j$
converge in law to chordal $\sle_\ka$ with respect to $\dstrong$.
\end{theorem}

This theorem follows from a series of more general results for
deterministic and random curves that we state formally in Section \ref
{subsec:main_result} (see Corollary~\ref{cor:app_sle}; a stronger
result applies when $\kappa\leq4$; see Corollary \ref
{cor:app_sle_simple}). It tells us in particular that we do not need to
know a priori that the laws of the random paths have subsequential weak
limits with respect to the uniform metric. This kind of a priori
pre-compactness has been obtained for some models: for example,
Kemppainen and Kemppainen and Smirnov
\cite{kemppainen,kemppainensmirnov} give a sufficient pre-compactness
criterion based on crossing probability estimates and the arguments in
\cite{aizenmanburchard}. However, these estimates require extra work
and are nontrivial in general. The Loewner driving function convergence
that we do require can be derived (e.g., via the recipe used in
\cite{lswlerw,schrammsheffieldexplorer,schrammsheffieldgff}) as soon
as one has sufficient control of an approximately conformally invariant
``martingale observable.'' Establishing and properly estimating these
observables has been the most difficult step in the proofs obtained
thus far, but at least we can now say that (for models with a built-in
time-reversal symmetry) this step is sufficient.

As a somewhat less technical motivation for our work, we note that part
of the appeal of $\sle$ theory is its supposed ``universality''---the
idea that $\sle$ is somehow \textit{the} canonical scaling limit of the
random self-avoiding paths that appear in critical two-dimensional
statistical physics. Although existing $\sle$ convergence proofs apply
only in very specific contexts, one can argue that the more we replace
the model-specific arguments in these proofs by general ones, the more
evidence we have for (some sort of) universality.

In this section we will begin by reviewing the Loewner evolutions; we
then define some useful metrics on curves and state both deterministic
and random versions of our main result. In Section \ref{sec:counterex}
we present a series of counterexamples, showing that the hypotheses in
the deterministic version of our convergence theorem are in fact
necessary. In Section \ref{sec:driving_conv} we state some known
consequences of driving function convergence and prove some auxiliary
lemmas. In Section \ref{sec:unif_conv} we prove our main result for
deterministic curves, and in Section \ref{sec:unif_conv_random} we give
the extension to random curves. Finally, in Section \ref{sec:app_sle}
we describe the application of our result to the $\sle_\ka$ processes
for $\ka<8$.

\subsection{Loewner evolutions}
\label{subsec:loewner_review}

Let $\HH$ be the upper half plane. We have chosen to use $\HH$ as our
canonical domain (mapping all other paths into $\HH$) because it is the
most convenient domain in which to define chordal Loewner evolutions.
However, we will also consider radial Loewner evolutions which are most
conveniently defined on the unit disc $\D$, and we will use the Cayley
transform $\vph(z)\defeq\f{z-i}{z+i}$ to easily go back and forth
between the two domains. To make the completion of $\HH$ a compact
metric space, we will endow $\HH$ with the metric it inherits from $\D$
via the map $\vph$: namely, we will let $\cdist(\cdot,\cdot)$ denote
the metric on $\HH$ given by
\[
\cdist(z,w)\defeq|\vph(z)-\vph(w)|
\]
and write $\overline{\HH}$ for the completion of $\HH$ with respect to
$\cdist$ (equivalently, its closure in~$\hat\C$). The map $\vph$ gives
an isometry of $\overline{\HH}$ with $\overline{\D}$. If $z\in\HH
\cup
\R$,
then $\cdist(z_n,z)\to0$ is equivalent to $|z_n-z|\to0$, and $\cdist
(z_n,\infty)\to0$ is equivalent to $|z_n|\to\infty$.

We now briefly review the Loewner evolutions, beginning with the
chordal version (for a more detailed account see
\cite{ahlfors,marshallrohde,lawler}). Suppose $\gam\dvtx[0,1]\to\overline
{\HH}
$ is a~continuous simple path starting at $\gam(0)=0$ and traveling in
$\overline{\HH}$, with \mbox{$\gam(t)\in\HH$} for all $t\in(0,1)$. For each
$t\in
[0,1)$, there is a unique conformal equivalence $g_t\dvtx\HH\setminus\gam
[0,t]\to\HH$ satisfying the so-called \textit{hydrodynamic normalization}
at $\infty$,
\[
\lim_{z\to\infty} [g_t(z)-z]=0.
\]
The quantity
\[
\f{1}{2} \lim_{z\to\infty} z[g_t(z)-z]
\]
is called the \textit{half-plane capacity} of $\gam[0,t]$ (w.r.t. $\infty
$), denoted $\capacity_\infty\gam[0,t]$. It is real and (strictly)
monotone increasing in $t$. Schramm's version of Loewner's theorem
states that if $\gam$ is reparametrized so that $\capacity_\infty
\gam
[0,t]=t$, then the maps $g_t$ satisfy the \textit{chordal Loewner equation},
%
%
\begin{equation} \label{eq:chordal_LDE}
\dot g_t(z)=\f{2}{g_t(z)-W_t},\qquad g_0(z)=z,
\end{equation}
where $W_t=g_t(\gam(t))$. Since $\gam(t)$ is not in the domain of $g_t$
it needs to be checked that $W_t$ can be defined as a limit; this is
done, for example, in \cite{lawler}, Lemma 4.2. The function $W_t$ is
continuous in $t$ and defined for all finite $t$ with $t<\capacity
_\infty\gam$, and is referred to as the (\textit{chordal}) \textit{driving
function} of $\gam$. To avoid ambiguity we will write from now on
$g_{\infty,t}\defeq g_t$, $W_{\infty,t}\defeq W_t$, and we continue to
work with the parametrization of $\gam$ defined on $[0,1]$ (rather than
with the Loewner capacity parametrization). For clarity of exposition
we will impose the technical condition that $\capacity_\infty\gam
[0,1]\incto\infty$ as $t\incto1$.

We now describe the radial Loewner evolution, which is more
conveniently defined in the unit disc $\D$. Again, suppose $\gam
\dvtx[0,1]\to\overline{\D}$ is a continuous simple path starting at
$\gam
(0)=1$ and traveling in $\overline{\D}$, with $\gam(t)\in\D
\setminus
\{0\}$
for all $t\in(0,1)$. For each $t\in[0,1)$ we now choose $g_t$ to be the
unique conformal map $\D\setminus\gam[0,t]\to\D$ with $g_t(0)=0$ and
$g_t'(0)>0$. The quantity $\log g_t'(0)$ is denoted $\capacity\gam
[0,t]$; if $\gam(1)=0$, then $\capacity\gam[0,t]\incto\infty$ as
$t\incto1$. Loewner's theorem states in this case that under the
parametrization $\capacity\gam[0,t]=t$, the maps $g_t$ satisfy the
\textit{radial Loewner equation},
%
%
\begin{equation} \label{eq:radial_LDE}
\dot g_t(z) = g_t(z) \f{e^{i W_t} + g_t(z)}{e^{i W_t} - g_t(z)},\qquad
g_0(z)=z,
\end{equation}
where $e^{iW_t}=g_t(\gam(t))$. Again this is continuous in $t$ and
defined for all finite~$t$ with $t<\capacity\gam$, and we will refer to
it as the (\textit{radial}) \textit{driving function} of $\gam$.

We note that it can be shown (using Schwarz reflection, see
\cite{lawler}, Section~4.1) that the Loewner differential equations (\ref
{eq:chordal_LDE}) and (\ref{eq:radial_LDE}) extend to points on the
boundary of the domain minus the starting point of the curve.

We can also try to reverse the above procedure: given a continuous
function $W_{\infty,t}$, we can solve (\ref{eq:chordal_LDE}) to obtain
the chordal Loewner maps $g_{\infty,t}$. For each $z\in\overline{\HH}$,
$g_{\infty,t}(z)$ is well defined up to the time that it collides with
$W_{\infty,t}$. Define the \textit{filling process} by
\[
K_{\infty,t}=\{z\in\overline{\HH}\dvtx g_{\infty,t}(z) \mbox{ not
defined at time } t\}
\]
and set $\HH_{\infty,t}=\HH\setminus K_{\infty,t}$. The question then
is whether there exists a curve~$\gam$ which \textit{generates} this
process, that is, such that for some parametrization of~$\gam$, $\HH
_{\infty,t}$ is the unique unbounded component of $\HH\setminus\gam
[0,t]$ for all $t$. We can do the same in the radial case (in the unit
disc), where we will denote the fillings by $C_t$ and set $\D_t=\D
\setminus C_t$. It is well known (see, e.g., \cite{lawler}) that there
exist continuous driving functions which give rise to filling processes
that are not generated by any curve, and it is trivial to construct a
curve which cannot arise from a continuous driving function (e.g., a
curve that retraces itself).

The definitions of $\capacity$, $C_t$, $\D_t$, $g_t$ and~$W_t$ (and
$\capacity_\infty$, $K_{\infty,t}$, $\HH_{\infty,t}$, $g_{\infty,t}$
and~$W_{\infty,t}$) above can be easily transferred to other (simply
connected) domains via conformal mapping. In particular, since we are
interested in curves traveling in $\overline{\HH}$, we will define a
capacity in $\HH$ \textit{with respect to $i$} by $\capacity_i
K=\capacity
\vph K$. We define a filling process with respect to $i$ by
$K_{i,t}(\gam)=\vph^{-1} C_t(\vph\gam)$, and we also write $\HH
_{i,t}=\HH\setminus K_{i,t}$. We define a driving function with respect
to~$i$ by $W_{i,t}(\gam)=W_t(\vph\gam)$, and we define a radial Loewner
chain $(g_{i,t})_t$ for~$\gam$ with respect to $i$ by $g_{i,t}=\vph
^{-1}\circ g_t\circ\vph$, where $(g_t)$ is the standard radial Loewner
chain corresponding to $\vph\gam$ [i.e., $(g_t)$ solves (\ref
{eq:radial_LDE}) with the driving function~$W_{i,t}$]. Similarly, for
general $x\in\HH$, we define $\capacity_x$, $K_{x,t}$, $\HH_{x,t}$,
$g_{x,t}$ and $W_{x,t}$ for $\gam$ via the unique automorphism $\psi_x$
of $\HH$ with $\psi_x(x)=i$, $\psi_x'(x)>0$. In particular, $\HH_{x,t}$
is the unique component of $\HH\setminus\gam[0,t]$ (where $\capacity
_x\gam[0,t]=t$) containing $x$, and $g_{x,t}$ is the unique conformal
map $\HH_{x,t}\to\HH$ which fixes $x$ and has $g_{x,t}'(x)>0$.

We can make similar definitions for the chordal case: in what follows,
we will generally consider curves traveling in $\overline{\HH}$ between
$-1$ and $1$, so we will let
\[
\psi_1 \dvtx z \mapsto\f{z+1}{z-1},\qquad
\psi_{-1} \dvtx z \mapsto\f{z-1}{z+1},
\]
so $\psi_1$ is a conformal automorphism of $\HH$ taking $1\mapsto
\infty
$ and $-1\mapsto0$, and $\psi_{-1}$ is a conformal automorphism of
$\HH
$ taking $-1\mapsto\infty$ and $1\mapsto0$. (We will often use $-1$ and
$1$ as endpoints---instead of $0$ and $\infty$---because it makes the
symmetry between the two parametrization directions slightly more
apparent.) For all other $x\in\R$ we let $\psi_x$ denote the unique
conformal automorphism of $\HH$ taking $x\mapsto\infty$ and fixing
$\{
\pm1\}$. Through the maps $\psi_x$ and using the chordal Loewner
evolution in the upper half-plane we can define $\capacity_x$,
$K_{x,t}$, $\HH_{x,t}$, $g_{x,t}$ and $W_{x,t}$ for $\gam$ traveling
from $-1$ to $1$ exactly as in the radial case.

\subsection{Families of curves}
\label{subsec:families}

We regard curves as continuous, nonlocally constant functions $f
\dvtx[0,1] \to\C$ (with respect to $\cdist$), taken modulo time
repara\-metrization: if $f_1,f_2\dvtx[0,1]\to\C$, we will say that the $f_i$
are \textit{the same up to reparametrization}, denoted $f_1\sim f_2$, if
there exists a continuously increasing bijection $\phi\dvtx[0,1]\to[0,1]$
such that $f_2=f_1\circ\phi$. A (\textit{directed}) \textit{curve} $\gam$ is then
defined to be an equivalence class modulo $\sim$. We often abuse
notation and write $\gam$ when we mean a particular parametrization of
$\gam$; to indicate the latter meaning we write $\gam\dvtx[0,1]\to\C$. We
write $\gam^-$ for the time reversal of~$\gam$. For any two curves
$\eta
_1,\eta_2$ such that the terminal point of $\eta_1$ is the initial
point of $\eta_2$, we will let $\eta_1\eta_2$ denote the concatenation
of these two curves. We will also use the notation $\gam[0,t]$ to
denote both the set $\gam([0,t])$ and the curve $\gam$ run up to time
$t$; the meaning should be clear from context.

Now let $\gam\dvtx[0,1]\to\overline{\HH}$ be a curve traveling between $-1$
and $1$ (in either direction), such that $\gam$ does not reach its
terminal point before time \mbox{$t=1$}. We will say that $\gam$ is
\textit{continuously driven with respect to $x$} if it arises from a
continuous driving function with respect to $x$. (A curve $\gam$ will
be continuously driven with respect to $x$ if its filling process
$K_{x,t}$ is continuously increasing; see \cite{lawler}, Section 4.1.)
We will say simply that $\gam$ is \textit{continuously driven} if it is
continuously driven with respect to its terminal point: such a curve
does not return into regions which are ``cut off'' from the terminal
point by $\gam$. If $\gam$ is continuously driven, then for any $x\in
\overline{\HH}$ which does not lie on $\gam$, $\gam$ can be parametrized
according to $\capacity_x$ up to time $\capacity_x\gam$, that is, up to
the infimum of times $t$ such that the point $x$ and the terminal point
of $\gam$ no longer lie in the closure of the same component of $\HH
\setminus\gam[0,t]$. In this case the reparametrized filling process of
$\gam$ corresponds to the curve $\tilde\gam$ which is the curve
$\gam$
stopped at time $\capacity_x\gam$, and $\tilde\gam$ is continuously
driven with respect to $x$. Moreover, $\tilde\gam$ is precisely the
entire portion of $\gam$ which is ``harmonically visible from $x$'':
after $\tilde\gam$ is traveled, a region containing $x$ is cut off and
$\gam$ does not re-enter this region. Thus every closed initial segment
of $\gam$ will be visible to some $x\in\HH\setminus\gam$, which does
not necessarily hold if $\gam$ is not continuously driven. Finally, we
will say that a curve is \textit{bidirectionally continuously driven} if
both $\gam$ and its time reversal $\gam^-$ are continuously driven.

We restrict our consideration to continuously driven curves traveling
between $-1$ to $1$ in $\overline{\HH}$ (this includes curves with
boundary intersections and self-intersections). It will be useful to
fix a countable dense subset $\Psi$ of~$\HH$; we then let
$\pathspaceR
=\pathspaceR_\Psi$ (resp., $\pathspaceL=\pathspaceL_\Psi$) denote the
space of all directed, continuously driven curves traveling from $-1$
to $1$ (resp., $1$ to $-1$) which avoid~$\Psi$. We let $\pathspace= \{
\gam\in\pathspaceR\dvtx\gam^-\in\pathspaceL\}$ denote the space of
bidirectionally continuously driven curves traveling from $-1$ to $1$.

If $\gam\in\pathspaceR$, we will let $\HH_t(\gam)\defeq\HH
_{1,t}(\gam
)$, $K_t(\gam)\defeq K_{1,t}(\gam)$ and so on. For $x\in\Psi$, we will
let $\tau_x=\tau_x(\gam)$ denote the infimum of times $t$ (under the
$\capacity_1$ parametrization) such that $x$ does not lie in the
closure of $\HH_t(\gam)$; that is, $\tau_x$ is the first time that $x$
is cut off from $1$ by $\gam$. If two curves $\gam_1,\gam_2\in
\pathspaceR$ agree for all times up to $\tau_x(\gam_1)$ [in which case
$\tau_x(\gam_1)=\tau_x(\gam_2)$], we will say that they are
\textit{equivalent viewed from $x$}, and write $\gam_1\sim_x\gam_2$.

We let $\simplepathsR$ denote the subspace of curves $\gam$ traveling
from $-1$ to $1$ such that $\gam$ is simple and boundary-avoiding. We
likewise define $\simplepathsL$ and $\simplepaths$; clearly these three
spaces are equivalent. For $\gam\in\pathspaceR$ parametrized by
$\capacity_1$ (or $\gam\in\pathspaceL$ parametrized by $\capacity
_{-1}$), we will say that $t$ is a \textit{disconnecting time} if $\gam
[0,t]\cap\gam[t,\infty)$ is totally disconnected, that is, has no
nontrivial connected components. We say that $\gam$ is
\textit{time-separated} if every time is a disconnecting time, and we let
$\tspathsR$ denote the subspace of curves $\gam\in\pathspaceR$ which
are time-separated. (This definition will be motivated later: see
Example \ref{ex:spikes} and Figure \ref{fig:spikes}. We remark that it
is easy to see that space-filling curves are not time-separated,
although they may be continuously driven.) Note the trivial inclusions
$\simplepathsR\hookrightarrow\tspathsR\hookrightarrow\pathspaceR$. We
make all these definitions symmetrically for $\gam\in\pathspaceL$, and
we let $\tspaths=\{\gam\in\tspathsR\dvtx\gam^-\in\tspathsL\}$ denote the
space of time-separated curves which are bidirectionally continuously driven.

\subsection{Metrics on the space of curves}
\label{subsec:curve_metrics}

In this section we introduce the distance functions which we will
consider on the space of curves. For two compact sets $A,B\subset\C$,
we have the $\cdist$-induced \textit{Hausdorff distance}
\[
\cdisthaus(A,B)
\defeq\inf\{\ep>0 \dvtx A \subset\cN(B,\ep) \mbox{ and }
B \subset\cN(A,\ep) \},
\]
where $\cN(A,\ep)$ denotes the open $\ep$-neighborhood of $A$ with
respect to the~met\-ric~$\cdist$, that is, $\cN(A,\ep) = \bigcup
_{a\in A} B_{\ep}(a)$, where \mbox{$B_{\ep}(a)\defeq\{z\in\C
\dvtx\cdist
(z,a)<\ep\}$}. For example, for two curves traveling in a metric space,
we can measure their proximity by the Hausdorff distance between their
image sets. If $\cH(\overline{\HH})$ denotes the set of all nonempty
compact subsets of $\overline{\HH}$ (with metric $\cdist$), then~$\cdisthaus$ makes $\cH(\overline{\HH})$ into a compact metric space.
However, most often we are interested in a finer notion of proximity
for curves which takes into account the order in which points are
visited. We therefore define a distance\vadjust{\goodbreak} function on the space of curves by
%
%
\begin{equation} \label{eq:unif_metric_curves}
\dstrong(\gam_1,\gam_2)
\defeq\inf_\phi\Bigl[
\sup_{0 \le t \le1}
\cdist\bigl(f_1\circ\phi(t), f_2(t)\bigr)
\Bigr],
\end{equation}
where $f_i$ is any function in the equivalence class $\gam_i$, and the
infimum is taken over all reparametrizations $\phi$ which are
continuously increasing \textit{bijections} of $[0,1]$. It can be checked
that $\dstrong$ is well defined and gives a metric on the space of
curves. We will refer to $\dstrong$ as the \textit{uniform metric}, and to
the topology it generates as the \textit{uniform topology}.

Our goal is to deduce convergence in this uniform topology from driving
function convergence. For $x\in\overline{\HH}$, denote by $d_{x}^{R}$
the distance function on $\pathspaceR$ which is defined by
%
%
\begin{equation} \label{eq:driving_metric}
d_{x}^{R}(\gamma_1, \gamma_2) \defeq|T_1 - T_2| +
\|W_{x,t\wedge T_1}(\gamma_1) - W_{x,t\wedge T_2}(\gamma_2)
\|_\infty,
\end{equation}
where $T_j\defeq\capacity_x\gamma_j$. (For $x\in\HH$ we use the radial
driving functions; for \mbox{$x\in\R$} we use the chordal versions.) Observe
that $d_{x}^{R}(\gam_1,\gam_2) = d_{i}^{R}(\psi_x\gam_1,\psi
_x\gam
_2)$, and that distinct paths $\gamma_1,\gamma_2\in\pathspaceR$ have
$d_{x}^{R}(\gam_1,\gam_2) = 0$ if and only if $\gam_1\sim_x\gam_2$.
It follows easily that $d_{x}^{R}$ is a metric on $\pathspaceR/\sim
_x$; in a slight abuse of language we will say that $\gam^j$ converges
to $\gam$ with respect to $d_{x}^{R}$ in $\pathspaceR$ if
$d_{x}^{R}(\gam^j,\gam)\to0$, that is, if convergence holds in
$\pathspaceR
/\sim_x$. We define similarly, for each $x$, the distance function
$d_{x}^{L}$ on $\pathspaceL$. Finally, if $\gam^j,\gam\in
\pathspaceR
$, their driving functions~$W_t^j,W_t$ with respect to the terminal
point $1$ are defined for all $t\ge0$. We let~$\dcapf$ be a metric on
$\pathspaceR$ such that $\dcapf(\gam^j,\gam)\to0$ if and only if
$W_t^j$ converges uniformly to $W_t$ on bounded intervals; we leave it
to the reader to verify that such a metric can be constructed. We
define likewise $\dcapb$ on $\pathspaceL$.

\subsection{Main result}
\label{subsec:main_result}

$\!\!\!$We now describe the main results of this paper. Throughout we will let
$(\gam^j)$ denote a sequence in $\simplepaths$, as is the case in
applications of interest. Our main deterministic result is the following:
\begin{theorem} \label{thm:unif_conv}
Let $\Psi$ be any countable dense subset of $\HH$, and let $(\gam^j)$
be a sequence in $\simplepaths$ such that for every $x\in\Psi$, we have
\[
\lim_{j\to\infty} d_{x}^{R}(\gam^j,\eta^x)
= \lim_{j\to\infty} d_{x}^{L}(\gam^{j-},\xi^x)
= 0
\]
for some fixed $\eta^x\in\tspathsR$, $\xi^x\in\tspathsL$. Then there
exists a curve $\gam\in\tspaths$ such that $\gam^j\to\gam$ with respect
to $\dstrong$. Moreover each $\hat\eta^x\defeq\eta^x[0,\tau
_x(\eta^x)]$
is an initial segment of $\gam$ while each $\hat\xi^x\defeq\xi
^x[0,\tau
_x(\xi^x)]$ is a concluding segment (up to the inclusion of endpoints),
and $\gam=\bigcup_x\hat\eta^x=\bigcup_x\hat\xi^x$ (up to the inclusion
of endpoints), where $\bigcup_x \hat\eta^x$ means the minimal curve of
which each $\hat\eta^x$ is an initial segment.
\end{theorem}
\begin{rmk}
In the theorem above, no a priori compatibility of the~$\eta^x,\xi^x$
is assumed. Note that according to our definitions of $\tspathsR$ and
$\tspathsL$, $\eta^x$ and $\xi^x$ travel between $-1$ and $1$, but are
uniquely specified only up to~$\sim_x$ (and thus are represented by
their initial segments $\hat\eta^x,\hat\xi^x$ stopped at the swallowing
time of $x$).
\end{rmk}

A substantially simpler criterion can be applied in the case when the
limiting curve is simple:
\begin{ppn} \label{ppn:unif_conv_simple}
Let $(\gam^j)$ be a sequence in $\simplepaths$ such that $\dcapf
(\gam
^j,\eta)\to0$ and $\dcapb(\gam^j,\xi)\to0$ for $\eta\in
\simplepathsR$,
$\xi\in\simplepathsL$. Then $\eta=\xi^-\defeqr\gam$ and $\gam
^j\to\gam
$ with respect to $\dstrong$.
\end{ppn}

For the general (nonsimple) case, Section \ref{sec:counterex} contains
a list of examples which show that the hypotheses in Theorem
\ref{thm:unif_conv} are necessary. We will exhibit the following:
\begin{longlist}[Example 2.5]
\item[Example \ref{ex:fwd_insufficient}.] $\gam\in\simplepaths$,
$d_{x}^{R}(\gam^j,\gam)\to0$ for all $x \in\Psi$, but $(\gam^j)$ not
$\dstrong$-Cauchy.\vspace*{2pt}
\item[Example \ref{ex:bidir.insuff}.] $\gam\in\tspaths$, $\dcapf
(\gam
^j,\gam)\to0$ and $\dcapb(\gam^{j-},\gam^-)\to0$, but $(\gam^j)$ not
$\dstrong$-Cauchy.
\item[Example \ref{ex:spikes}.] $\gam\in\pathspace$,
$d_{x}^{R}(\gam
^j,\gam)\to0$ and $d_{x}^{L}(\gam^{j-},\gam^-)\to0$ for all $x
\in
\Psi$, but~$(\gam^j)$ not $\dstrong$-Cauchy.
\item[Example \ref{ex:bidir.dense.insuff_ctsdriven}.] $\gam_1\in
\pathspaceR$, $\gam_2\in\pathspaceL$, $\gam_1\ne\gam_2$, but
$d_{x}^{R}(\gam^j,\gam_1)\to0$ and $d_{x}^{L}(\gam^{j-}$, $\gam
^-_2)\to0$
for all $x \in\Psi$.
\end{longlist}
Example~\ref{ex:fwd_insufficient} is a well-known example (essentially
the same as \cite{lawler}, Examp\-le~4.49) which shows that even in the
case that $\gam$ is simple without boundary intersections, one cannot
replace the $d_{x}^{R}$ and $d_{x}^{L}$ convergence (for all
$x\in
\Psi$) required in Theorem~\ref{thm:unif_conv} with $d_{x}^{R}$
convergence alone. The other examples are new to this paper.
Example \ref{ex:bidir.insuff} shows that in the first half of
Theorem \ref{thm:unif_conv}, for~$\gam$ with boundary intersections and
self-intersections permitted, one cannot replace $d_{x}^{R}$ and
$d_{x}^{L}$ convergence (for all $x \in\Psi$) with $\dcapf$ and
$\dcapb$ convergence. It is indeed necessary to consider points $x$
other than the two endpoints of the path. (We remark, however, that
$d_{x}^{R}$ convergence to $\eta^x$ automatically implies $d_{x'}^{R}$
convergence to $\eta^x$ whenever $x$ and $x'$ lie in the same
component of $\HH\setminus\eta^x$; thus it is enough for $\Psi$ to
include one $x$ in each component of $\HH$ minus the Hausdorff limit of
the $\gamma^j$, provided that the union of these components is dense.)
Example \ref{ex:spikes} shows that, in the first half of Theorem
\ref{thm:unif_conv}, one cannot replace $\tspathsR$ and $\tspathsL$ with
$\pathspaceR$ and $\pathspaceL$; that is, one cannot remove the
time-separation condition. Finally, Example \ref
{ex:bidir.dense.insuff_ctsdriven} shows that without the
time-separation condition in Theorem \ref{thm:unif_conv}, it is
possible for the $d_{x}^{R}$ limits to be incompatible with the
$d_{x}^{L}$ limits. This indicates that some care will be required to show
that the $\eta^x$ and $\xi^x$ in Theorem \ref{thm:unif_conv} are
compatible with one another, and that they uniquely determine $\gam$ in
the sense described.

Readers with a fondness for puzzles may attempt to construct these
examples themselves before reading Section \ref{sec:counterex}. Readers
with limited time or patience for examples may proceed directly to
Section \ref{sec:driving_conv}; the remainder of the paper is logically
independent of Section \ref{sec:counterex}. Using standard topological
arguments, we extend Theorem \ref{thm:unif_conv} to random curves in
Section \ref{sec:unif_conv_random}.
\begin{theorem} \label{thm:unif_conv_random}
Let $(\gam^j)$ be a sequence of random curves in $\simplepaths$ such
that for every $x\in\Psi$, the $\gam^j$ (resp., $\gam^{j-}$) converge
in law with respect to~$d_{x}^{R}$ (resp.,~$d_{x}^{L}$) to a~random
curve $\eta^x\in\tspathsR$ (resp., $\xi^x\in\tspathsL$). Then the
$\gam
^j$ converge in law to a random curve $\gam\in\tspaths$ with respect to
$\dstrong$. This $\gam$ can be coupled with the curves $\hat\eta
^x\defeq
\eta^x[0,\tau_x(\eta^x)]$ and $\hat\xi^x\defeq\xi^x[0,\tau
_x(\xi^x)]$
in such a~way that each $\hat\eta^x$ is an initial segment of $\gam$,
each $\hat\xi^x$ is a concluding segment, and $\gam=\bigcup_x\hat
\eta
^x=\bigcup_x\hat\xi^x$ up to the inclusion of endpoints.
\end{theorem}

In Section \ref{sec:app_sle} we will see that this result applies in
particular to the case that $\gam$ is a (chordal) $\sle_\ka$ for some
$\ka<8$.
\begin{cor} \label{cor:app_sle}
Let $(\gam^j)$ be a sequence of random curves in $\simplepaths$ such
that for every $x\in\Psi$, the $\gam^j$ (resp., $\gam^{j-}$) converge
in law with respect to $d_{x}^{R}$ (resp., $d_{x}^{L}$) to $\sle
_\ka
$ (for $\ka<8$) traveling from $-1$ to $1$ (resp., from $1$ to $-1$) in
$\overline{\HH}$. Then the $\gam^j$ converge in law to $\sle_\ka$ with
respect to $\dstrong$. Furthermore, this implies that $\sle_\ka$ is
time reversible (for this particular value of $\ka$), that is, that the
law of the time-reversal of an $\sle_\ka$ from $-1$ to $1$ is an
$\sle
_\ka$ from $1$ to $-1$.
\end{cor}

Our results indicate a general method for proving uniform convergence
of discrete curves to $\sle$: if one can establish $d_{x_0}^{R}$
convergence for a sequence of random paths with respect to an arbitrary
fixed interior point $x_0$, this immediately implies $d_{x}^{R}$
convergence with respect to a countable dense collection of fixed
interior points $x\in\Psi$; if one also proves $d_{x_0}^{L}$
convergence (again for~$x_0$ generic), then Corollary~\ref{cor:app_sle}
yields the desired convergence in law with respect to $\dstrong$. It
was proven in \cite{zhan} that the time reversal of $\sle_\ka$ is again
$\sle_\ka$ for $\ka\le4$, and the same is believed true for $4<\ka<8$
but is not known. Nevertheless, we expect Corollary \ref{cor:app_sle}
to apply in cases where the symmetry of the $\gam^j$ and $\gam^{j-}$ is
intrinsic to the model. (Examples include the Ising model spin
interfaces, the FK cluster boundaries, the percolation interfaces and
the level lines of the Gaussian free field.) If a discrete model did
not have such a time-reversal symmetry---and one only had direct access
to the driving functions for one parametrization direction (as is the
case, e.g., for the harmonic explorer
\cite{schrammsheffieldexplorer})---one could in principle use
Theorem~\ref{thm:unif_conv_random} to prove convergence to $\sle_\ka$ without first
proving (or in the process establishing) a~reversibility result:
\begin{cor} \label{cor:app_sle_generalized}
Let $(\gam^j)$ be a sequence of random curves in $\simplepaths$ such
that for every $x\in\Psi$, the $\gam^j$ converge in law with respect to
$d_{x}^{R}$ to $\sle_\ka$ (for $\ka<8$) traveling from $-1$ to
$1$ in
$\overline{\HH}$, and the $\gam^{j-}$ have subsequential limits in law
with respect to $d_{x}^{L}$ which lie in $\tspathsL$. Then the
$\gam
^j$ converge in law to $\sle_\ka$ with respect to $\dstrong$.
\end{cor}

Finally, Proposition \ref{ppn:unif_conv_simple} gives the following
simplified criterion for convergence to $\sle_\ka$ when $\ka\le4$
(i.e., when the curve is a.s. in $\simplepaths$):
\begin{cor} \label{cor:app_sle_simple}
Let $(\gam^j)$ be a sequence of simple random curves in~$\pathspace$.
Let $\ka\le4$, and suppose that the $\gam^j,\gam^{j-}$ converge in law
(with respect to the~$\dcapf$ and $\dcapb$ metrics, resp.) to
$\sle_\ka$. Then the $\gam^j$ converge in law to $\sle_\ka$ with
respect to~$\dstrong$.
\end{cor}

\subsection{Outline of argument}

In this section we sketch the proof of Theorem~\ref{thm:unif_conv}. For
convenience, in what follows we let all curves started from $-1$
(resp.,~1) be parametrized by $\capacity_1$ (resp., $\capacity_{-1}$).

\textit{Step} 1: \textit{Construction of forward and reverse limiting curves $\eta
,\xi$.} Since no a priori compatibility among the $\eta^x$ or $\xi^x$
was assumed, the first step is to show that if we consider, say, the
forward direction, all the $\eta^x$ are consistent with one another and
with a single limiting curve which is their union in some sense. In
Section \ref{sec:driving_conv}, we will review the notion of Carath\'
eodory convergence, a known consequence of Loewner driving convergence.
Roughly speaking this will tell us that whenever $d_{x}^{R}(\gam
^j,\gam)\to0$, the filling processes of the $\gam^j$ with respect to
$x$ converge to the filling process of $\gam$ with respect to $x$. It
follows from this that for any $x,x'$, $\eta^x$ and $\eta^{x'}$ must
agree at least up to the first time $t$ that one of them is cut off
from the terminal point. Thus there is a unique half-open curve $\eta
\dvtx[0,1)\to\overline{\HH}$ with $\eta\sim_x\eta^x$ for all $x$, and
furthermore one can show that initial segments of the $\gam^j$ converge
in the Hausdorff sense to initial segments of $\eta$ (see Section \ref
{ssec:hausdorff}). Symmetrically we construct $\xi\dvtx[0,1)\to\overline
{\HH}$
from the $\xi^x$.

For simplicity we now restrict to the case where $\eta$ and $\xi$ can
be extended by continuity to closed curves which are simple and
boundary-avoiding.

\textit{Step} 2: \textit{Compatibility}: \textit{$\xi$ is the time reversal of $\eta$.} Let
$z_1=\eta(t_1)$ and $z_2=\eta(t_2)$ for $t_1<t_2$; we must show that
$\xi$ visits $z_2$ before $z_1$. The key is that the driving function
not only gives information about the shape of the filling, but also
about the location of the ``tip'' $\gam(t)$: for $\gam$ a continuously
driven curve and $z\in\HH$, we can use the driving function up to time
$t$ to deduce the probability that a Brownian motion started at $z$ and
stopping upon hitting $\gam\cup\R$ will be stopped to the left or right
of $\gam(t)$. We will show (Section \ref{ssec:hitting_brownian}) that
driving function convergence implies convergence of these Brownian
hitting probabilities.

Let $t_1<t_*<t_2$, and write $\eta_*=\eta[0,t_*]$, $\eta_*^j=\gam
^j[0,t_*]$, $\bar\eta_*=\eta[t_*,\infty)$ and $\bar\eta_*^j=\eta
^j[t_*,\infty)$. Let $\ep>0$ be such that $\eta$ does not re-enter
$\overline{B_{\ep}(z_1)}$ after time~$t_*$. Then there exists
$0<\ep
'<\ep$ sufficiently small so that for any \mbox{$z\in B_{\ep
'}(z_1)\setminus
\eta$}, a~Brownian motion started at $z$ and stopped upon hitting $\R
\cup
\eta$ has probability less than $\de$ (for $\de>0$ small) of being
stopped on $\bar\eta_*$. It then follows from the results of
Section \ref{ssec:hitting_brownian} that for large enough $j$ a
Brownian motion started at $z$ and stopped upon hitting $\R\cup\gam^j$
has probability less than $2\de$ of being stopped on $\bar\eta_*^j$.

On the other hand, the $\bar\eta_*^j$ are initial segments of the
$\gam
^{j-}$, and so must converge in the Hausdorff sense (at least along a
subsequence) to an initial segment of $\xi$ containing $z_2$. Thus if
$\xi$ visits $z_1$ before $z_2$, the $\bar\eta_*^j$ must get
arbitrarily close to $z_1$. This contradicts the observation above that
a Brownian motion started anywhere in $B_{\ep}(z_1)$ and stopped
upon hitting $\R\cup\gam^j$ has a very low probability of being stopped
on $\bar\eta_*^j$. It follows that $\eta=\xi^-\defeqr\gam$.

\textit{Step} 3: \textit{Uniform convergence.} It remains to show that the $\gam^j$
converge uniformly to $\gam$. With $z_1,z_2$ as above, let $\gam
^j[z_1,z_2]$ denote the portion of $\gam^j$ between its nearest
approach to $z_1$ and its nearest approach to $z_2$, with ties broken,
for example, by choosing the earlier time. It follows from the above
that for large enough $j$, the nearest approach to $z_1$ occurs before
the nearest approach to $z_2$. Thus, if we break up the curve $\gam$
into (finitely many) segments $\gam[t_i,t_{i+1}]$ of small diameter,
for large enough $j$ we obtain a corresponding partition of $\gam^j$
into segments $\gam^j[\gam(t_i),\gam(t_{i+1})]$. It then suffices to
show that any subsequential $\cdisthaus$ limit $B$ of $\gam^j[z_1,z_2]$
is contained in $\gam[t_1,t_2]$. By symmetry it suffices to show that
$B$ contains no point of $\gam[t_2,\infty)$, and this follows from the
arguments of Step 2, proving the result.

\section{Counterexamples}
\label{sec:counterex}

In the examples of this section, we consider families of curves
traveling in a domain $D$ between distinct boundary points~$a,b$, where
$(D;a,b)$ is not necessarily $(\HH;-1,1)$. Clearly, all definitions (of
spaces, metrics, etc.) in Section \ref{sec:intro} can be made
analogously for these families via a~conformal mapping $D\to\HH$ taking
$a\mapsto-1$ and $b\mapsto1$. We continue to use the notation
introduced above to refer to these newly defined objects.
\begin{ex}
\label{ex:fwd_insufficient}
We consider curves traveling between $0$ and $\infty$ in $\HH$. For
$n\in\N$, let $z_n = (-1)^n+in$, and let $w_n=in/2$. We will let
$\gam$
denote the curve which is a linear interpolation of the points
\[
0, z_1, w_1, z_2, w_2, \ldots.
\]
See Figure \ref{fig:fwd_insufficient}. Since $z_n\to\infty$ and
$w_n\to
\infty$, we see that $\gam$ is indeed a continuous simple curve from
%
%
\begin{figure}

\includegraphics{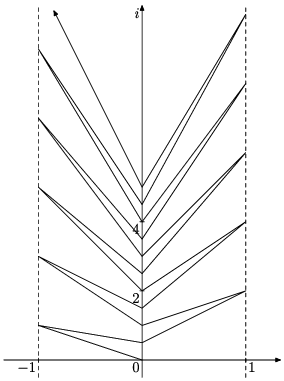}

\caption{Example \protect\ref{ex:fwd_insufficient}: beginning of curve
$\gam$.}
\label{fig:fwd_insufficient}
\end{figure}
$0$ to $\infty$. We then let $\gam^j=2^{-j}\gam$: it is easy to see
that as $\ep\to0$, the rescaled curves $\gam^j$ converge, both
with\vadjust{\goodbreak}
respect to $\cdisthaus$ and with respect to $d_{x}^{R}$ for any
$x\in
\HH$ off the imaginary axis, to the simple path that traces the
imaginary axis. However, it is clear that the $\gam^j$ have no
$\dstrong
$ limit.
\end{ex}

\begin{ex}
\label{ex:bidir.insuff}
We consider curves traveling chordally in $\overline{\D}$ between
$-1$ and~$1$. Let $\eta_i\dvtx[0,1]\to\C$ ($1\le i\le3$) be defined by
\begin{eqnarray*}
\eta_1(t)&\defeq&-1+(1+i)t, \\
\eta_2(t)&\defeq& i-2it, \\
\eta_3(t)&\defeq&-i+(1+i)t
\end{eqnarray*}
and let $\gam=\eta_1\eta_2\eta_3$ [Figure
\ref{fig:bidir.bd.intersect}(c)]; note that $\gam\in\tspaths$. We can easily
find a sequence $(\gam_1^j)$ in $\simplepaths$ converging uniformly to
$\gam$ [Figure \ref{fig:bidir.bd.intersect}(a)]. We can likewise find a
sequence $(\gam_2^j)$ in $\simplepaths$ converging uniformly to
$\tilde
\gam\defeq\eta_1\eta_2\eta_2^-\eta_2\eta_3$ [Figure \ref
{fig:bidir.bd.intersect}(b)]. But both the $\gam_1^j$ and the $\gam_2^j$
converge with respect to $\dcapf$ to~$\gam$, and with respect to $d^L$
to $\gam^-$. Letting $(\gam^j)$ be the sequence obtained by
interweaving $(\gam_1^j)$ and $(\gam_2^j)$, we have
\[
\lim_{j\to\infty}\dcapf(\gam^j,\gam)
=\lim_{j\to\infty}\dcapb(\gam^{j-},\gam^-)
= 0,
\]
but clearly $(\gam^j)$ is not $\dstrong$-Cauchy. Figure \ref
{fig:bidir.self.intersect} illustrates essentially the same example
%
%
\begin{figure}

\includegraphics{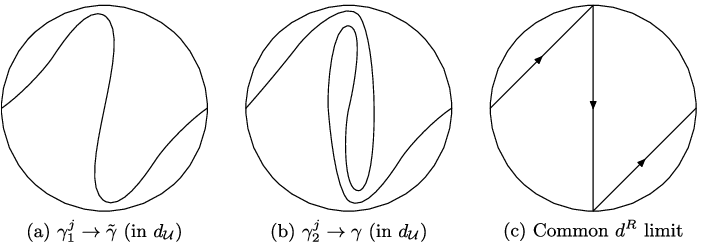}

\caption{Example \protect\ref{ex:bidir.insuff}: $\dcapf$ limit with
boundary intersections.}
\label{fig:bidir.bd.intersect}
\end{figure}
%
%
\begin{figure}[b]

\includegraphics{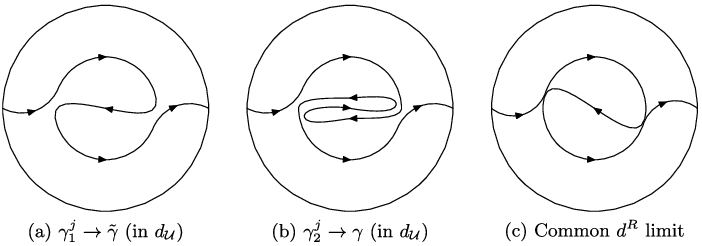}

\caption{Example \protect\ref{ex:bidir.insuff}: $\dcapf$ limit with
self-intersections.}
\label{fig:bidir.self.intersect}
\end{figure}
when the $\dcapf$ and $\dcapb$ limiting curves are allowed to have
self-intersections but not boundary intersections.
\end{ex}
\begin{ex}
\label{ex:spikes}
A useful construction for us is the curve $P$ which is formed by taking
the straight path from $0$ to $1$ and adding increasingly small,
mutually disjoint loops at the dyadic points; these loops are traveled
in the clockwise direction. To be more precise, begin with the straight
path~$P_0$ from~$0$ to $1$. For each $k\in\N$, let $t_{k,j}=2^{-k+1}j +
2^{-k}$ for $0\le j\le2^{k-1}$. Given~$P_{k-1}$, for each $j$ define a
small clockwise loop $\ell_{k,j}$ which begins and ends at~$t_{k,j}$
and otherwise is contained in $\HH$; the size of the $\ell_{k,j}$
should tend to zero in~$k$. Set $P_k$ to be $P_{k-1}$ with the $\ell
_{k,j}$ added, so that the time~$P_{k-1}$ spends on each
$(t_{k,j}-2^{-k},t_{k,j}+2^{-k})$ is divided in thirds between
$(t_{k,j}-2^{-k},t_{k,j})$, the loop $\ell_{k,j}$, and
$(t_{k,j},t_{k,j}+2^{-k})$. Figure \ref{fig:dyadic} shows the first few
iterations of this construction. We will refer to the limiting curve
$P$ as the ``dyadic loops curve based on $[0,1]$''; it is clear that we
can construct a dyadic loops curve based on any simple curve. If a
curve first traces $[0,1]$ and then traces backwards the path of diadic
loops, then all of the points on $[0,1]$ will be double points, but
there will be a dense collection of times mapping to nondouble points.

%
%
\begin{figure}

\includegraphics{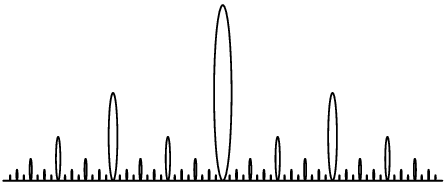}

\caption{Construction of dyadic loops curve.}
\label{fig:dyadic}
\end{figure}

Consider the simple curve shown in Figure \ref{fig:spikes}: first it
travels the left part of the curve from $-1$ to $-1+3i$, with loops to
the left and $U$-shaped ``hooks'' to the right; each ``hook'' is a path
%
%
\begin{figure}[b]

\includegraphics{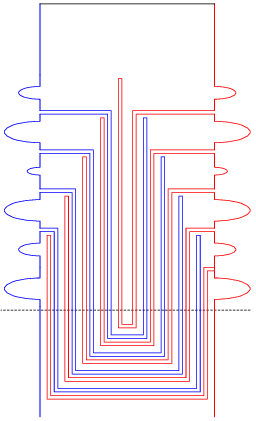}

\caption{Example \protect\ref{ex:spikes}.}
\label{fig:spikes}
\end{figure}
that passes below the dotted line, then returns upward to approximately
the place it started, then approximately retraces itself. The curve
then goes right to $1+3i$ and travels the right part of the curve from
$1+3i$ to~$1$, with loops to the right and hooks to the left; again,
all hooks are bent to pass below the dotted line. The hooks on the two
sides are interlocking: thus, for the curve traveled in either
direction, each successive hook is mostly ``harmonically enclosed''
within previous hooks.
\end{ex}

We define a sequence of curves $\gam^j\in\simplepaths$ which are
versions of this curve, so the hooks and loops become more numerous,
and the distance between the two vertical sides decreases to zero, as
$j\to\infty$. One then sees that for all~$x$ in a countable dense set,
$\gam^j$ converges with respect to both $d_{x}^{R}$ and $d_{x}^{L}$
to the dyadic loops curve $\gam$ which travels clockwise around the
boundary of the line segment between $0$ and $3i$. But it is clear that
the $\gam^j$ do not converge uniformly to $\gam$.

Before giving Example \ref{ex:bidir.dense.insuff_ctsdriven}, we present
here a simplified version:
\begin{ex}
\label{ex:bidir.dense.insuff}
Let $(\gam_1^j)$ be a sequence of simple curves such as the one shown
in Figure~\ref{fig:bidir.dense.insuff}(a), converging uniformly to
the curve $\gam$ shown in Figure~\ref{fig:bidir.dense.insuff}(b).
Write $\gam\defeq\eta_1\eta_2\eta_3\eta_4\eta_5\eta_6$ where the
numbering is as in the figure. Now define $\tilde\gam\defeq\eta
_1\eta
_2\eta_4\eta_3\eta_5\eta_6$, and let $(\gam_2^j)$ be a sequence of
simple curves converging uniformly to $\tilde\gam$. Let $(\gam^j)$ be
the sequence obtained by interweaving the $\gam_1^j$ and $\gam_2^j$. It
can be checked that the $\gam^j$ are Cauchy with respect to
$d_{x}^{R}$ and $d_{x}^{L}$ for every $x\in\Psi$, but they clearly are
%
%
\begin{figure}

\includegraphics{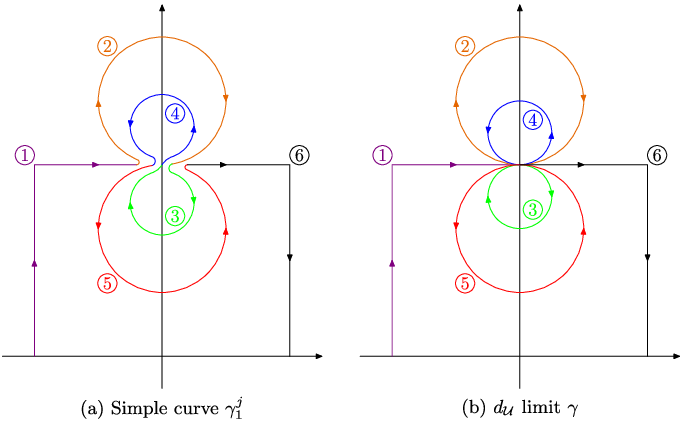}

\caption{Example \protect\ref{ex:bidir.dense.insuff}: incompatible
forward and reverse limits.}
\label{fig:bidir.dense.insuff}
\vspace*{-3pt}
\end{figure}
not $\dstrong$-Cauchy. In this example the $\gam^j$ do not converge
with respect to every $x\in\Psi$ to continuously driven curves (i.e.,
to limits in $\pathspaceR$ or $\pathspaceL$). For example, if $x$ lies
inside the uppermost inner loop $\eta_4$, then the $\gam^j$ converge in
driving function to the curve $\eta_1\eta_2\eta_4$, which does not lie
in $\pathspaceR$. (Nevertheless, this curve can be generated by a
continuous driving function \textit{with respect to $x$}.)\vspace*{-3pt}
\end{ex}
\begin{ex} \label{ex:bidir.dense.insuff_ctsdriven}
We now present a modification of Example \ref{ex:bidir.dense.insuff} in
which all $d_{x}^{R}$ and $d_{x}^{L}$ limits lie in $\pathspaceR$
and $\pathspaceL$, respectively.

%
%
\begin{figure}

\includegraphics{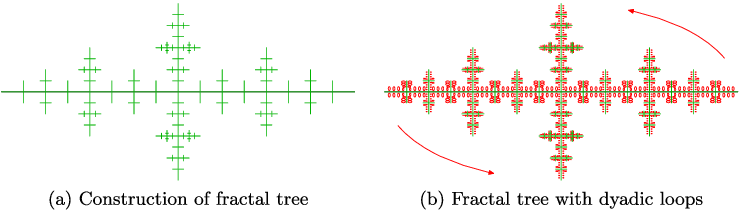}

\caption{Example \protect\ref{ex:bidir.dense.insuff}: incompatible
forward and reverse limits.}
\label{fig:fractal.tree}
\vspace*{-3pt}
\end{figure}

The main construction we will use is the ``fractal tree,'' the
beginning iterations of which are shown in Figure
\ref{fig:fractal.tree}(a). We leave it to the reader to verify that this
tree can be constructed so that the curve which traces its boundary
clockwise (i.e., traces the conformal boundary of the complement of the
tree) has a dense set of times mapping to double points and avoids a
countable dense subset of $\HH$. Moreover, if we then traverse the
boundary counterclockwise and add small loops beginning and ending at
the same prime end of the conformal boundary [Figure~\ref
{fig:fractal.tree}(b)], in the limit we will obtain a curve which is
continuously driven in the forward but not the reverse direction.
We\vadjust{\goodbreak}
will refer to the counterclockwise portion of the curve as the ``dyadic
loops curve based on the fractal tree.'' From now on, we will use the
diagram in Figure \ref{fig:fractal.tree}(b) to indicate this
limiting curve.

%
%
\begin{figure}[b]
\vspace*{-3pt}
\includegraphics{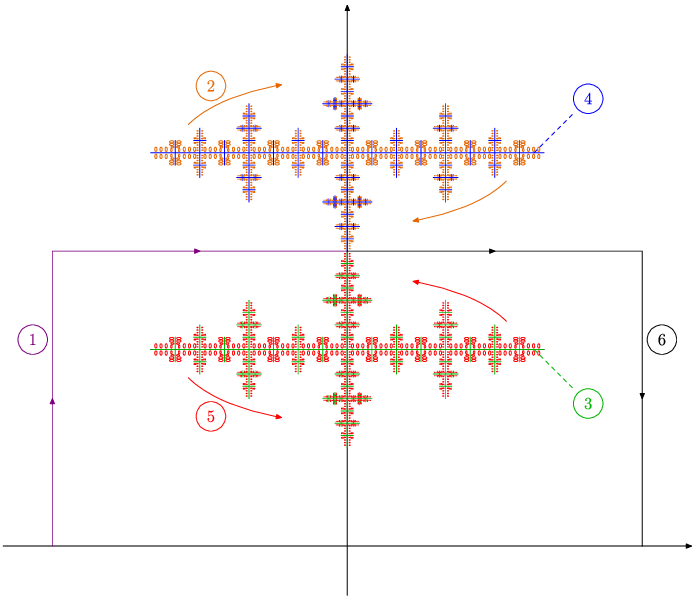}

\caption{Example \protect\ref{ex:bidir.dense.insuff_ctsdriven}:
$\dstrong$ limit $\gam$.}
\label{fig:bidir.dense.insuff_ctsdriven_2}
\end{figure}

Consider now the curve $\gam$ which is shown in Figure \ref
{fig:bidir.dense.insuff_ctsdriven_2}. It is the concatenation of $\eta
_i$ ($1\le i\le6$), where:
\begin{enumerate}
\item$\eta_1$ is the linear interpolation of the points $-1$,
$-1+3i$, $3i$;
\item$\eta_2$ is the clockwise dyadic loops curve based on the upper
fractal tree;\vadjust{\goodbreak}
\item$\eta_3$ travels the lower fractal tree clockwise beginning and
ending at $3i$;
\item$\eta_4$ travels the upper fractal tree counterclockwise
beginning and ending at $3i$;
\item$\eta_5$ is the counterclockwise dyadic loops curve based on the
lower fractal tree;
\item$\eta_6$ is the interpolation of the points $3i$, $1+3i$, $1$.
\end{enumerate}
Let $\Psi$ be a countable dense subset of $\HH$ avoiding $\gam$; we
leave it to the reader to verify that one exists. We then define
$\tilde
\gam\defeq\eta_1\eta_2\eta_4\eta_3\eta_5\eta_6$.

We can find simple curves $\gam_1^j,\gam_2^j$ converging uniformly to
$\gam,\tilde\gam$, respectively; by interweaving the sequences we obtain
a sequence $(\gam^j)$ which fails to converge uniformly. But we can
check that for all $x\in\Psi$, we have $d_{x}^{R}(\gam^j,\gam
_1)\to0$
where $\gam_1\defeq\eta_1\eta_2\eta_3\eta_5\eta_6$, and
$d_{x}^{L}(\gam^{j-},\gam_2^-)\to0$ where $\gam_2\defeq\eta_1\eta
_2\eta
_4\eta_5\eta_6$. We have $\gam_1\in\pathspaceR$ and $\gam_2\in
\pathspaceL$, but $\gam_2\ne\gam_1^-$ so the forward and reverse limits
are incompatible.
\end{ex}

\section{Driving function convergence}
\label{sec:driving_conv}

In this section we present (along with a~few related facts) a known
implication of $d_{x}^{R}$ convergence, namely Carath\'eo\-dory convergence.
\begin{rmk} All of the results in this section continue to hold if
$\Psi
$ is replaced with some $\Psi'$ which is the union of $\Psi$ with a
countable dense subset of $\overline\R\setminus\{-1,1\}$, and the
path spaces $\pathspaceR$, $\pathspaceL$, etc. are redefined
accordingly. When $x \in\R$, the metrics $d_{x}^{L}$ and $d_{x}^{R}$
correspond to chordal rather than radial Loewner driving functions.
\end{rmk}

We let $\initpathsR$ denote the space of all curves which can arise as
closed initial segments of curves in $\pathspaceR$; we define
$\initpathsL$ similarly. All the definitions of Section \ref{sec:intro}
(filling processes, distance functions, etc.) can be made for these
spaces in exactly the same way. In particular, if $\gam^j,\gam\in
\pathspaceR$ with $d_{x}^{R}(\gam^j,\gam)\to0$, then, under the
$\capacity_x$ parametrization, we have $d_{x}^{R}(\gam^j[0,t],\gam
[0,t])\to0$ as well for any $t$. The distance $d_{x}^{R}$ is a metric
on $\initpathsR/\sim_x$.
\begin{ppn} \label{ppn:separable}
For each $x \in\Psi$, the spaces $\pathspaceR/\sim_x$ and
$\pathspaceL/\sim_x$ are separable with respect to the topology
generated by $d_{x}^{R}$ and $d_{x}^{L}$, respectively; likewise
$\pathspaceR/\sim_1$ and $\pathspaceL/\sim_{-1}$ are separable with
respect to the topology generated by $\dcapf$ and $\dcapb$,
respectively. Also, $\pathspace$ is separable with respect to the
topology generated by $\dstrong$.
\end{ppn}
\begin{pf}
For the separability of $\pathspaceR/\sim_x$ it suffices to prove
separability of a larger metric space: for metric spaces separability
is equivalent to second-countability, and second-countability is
inherited in the subspace topology (see, e.g., \cite{munkres}).

It is easy to see that the space of all pairs $(W,T)$ where $T>0$ and
$W\dvtx\allowbreak[0,T]\to\R$ is continuous, is separable under the metric (\ref
{eq:driving_metric}): a countable dense set can be constructed by taking
$W$ continuous and linear (with rational derivative) on each of a
finite set of rational-length intervals which partition $[0,T]$, for
$T$ rational. Separability immediately follows for $\pathspaceR/\sim_x$
and $\pathspaceL/\sim_x$, and it follows for $\pathspaceR/\sim_1$, and
$\pathspaceR/\sim_{-1}$ by a similar argument. For the metric
$\dstrong
$ a countable dense subset of $\pathspace$ can similarly be given using
functions which are piecewise linear as maps into~$\HH$.
\end{pf}

\subsection{\texorpdfstring{Carath\'eodory convergence}{Caratheodory convergence}}
\label{ssec:caratheodory}

We begin by recording some preliminary consequences of $d_{x}^{R}$
convergence for a fixed $x\in\Psi$. Extending our previous notation,
if $x\in\Psi$ and $\gam$ in $\initpathsR$ or $\initpathsL$ is
parametrized according to $\capacity_x$, then $\HH_{x,t}=\HH
_{x,t}(\gam
)$ denotes the unique component of $\HH\setminus\gam[0,t]$ containing~$x$,
and $K_{x,t}=K_{x,t}(\gam)=\HH\setminus\HH_{x,t}$ denotes the
filling with respect to~$x$ at time~$t$. If $T=\capacity_x\gam$ then we
will generally drop the time subscript and simply write $\HH_x=\HH
_x(\gam),K_x=K_x(\gam)$. We let $\overline{\HH}_{x,t}$ denote the closure
of $\HH_{x,t}$ in $\overline{\HH}$ under the $\cdist$ metric.
\begin{defn}
Let $g^j = (g_t^j)_{0\le t\le T}, g = (g_t)_{0\le t\le T}$ be (radial
or chordal) Loewner chains with respect to $x$, defined on $\HH$. We
say that \textit{$g^j$ converges to $g$ in the Carath\'eodory sense with
respect to $x$}, denoted $g^j \caraconv_x g$, if for all $\ep>0$ and
$t\le T$, $g^j \to g$ uniformly on
\[
[0,t]\times\overline{\HH}{}^{(\ep)}_{x,t}\qquad
\mbox{where }
\overline{\HH}{}^{(\ep)}_{x,t} \defeq
\{z\in\overline{\HH}\dvtx \cdist(z,\overline{K}_{x,t}) \ge\ep\}.
\]
In particular, this implies that $g_t^j \to g_t$ pointwise on $\HH
_{x,t}$ for each $t$. Also, $\imag g_t^j(z)$ has a nonzero limit as
$j\to\infty$ if and only if $z \in\HH_{x,t}$.
\end{defn}

The following proposition relates driving function convergence to
Cara\-th\'eodory convergence. The result for chordal Loewner chains is
proved in~\cite{lawler}; the proof for radial Loewner chains is
entirely similar and we omit it here.
\begin{ppn} \label{ppn:cara_conv_loewner}
Let $g^j=(g_t^j)_{0\le t\le T}, g=(g_t)_{0\le t\le T}$ be Loewner
chains corresponding to the driving functions $W_t^j, W_t$, all with
respect to $x$. If \mbox{$W_t^j \to W_t$} uniformly on $[0,T]$, then $g^j
\caraconv_x g$.
\end{ppn}
\begin{cor} \label{cor:cara_from_driving}
Let~$\gam^j,\gam\in\initpathsR$ with $T^j=\capacity_x\gam^j$,
$T=\capacity_x\gam$ and let~$g^j,g$ denote the Loewner chains with
respect to $x$ corresponding to $\gam^j,\gam$, respectively. Suppose
$d_{x}^{R}(\gam^j,\gam)\to0$. Then $(g_s^j)_{0\le s\le
t}\caraconv_x(g_s)_{0\le s\le t}$ for all $t<T$. We also have
$g_{T_j}^j\to g_T$ uniformly on
$\overline{\HH}{}^{(\ep)}_x\defeq\overline{\HH}{}^{(\ep)}_{x,T}$.
\end{cor}
\begin{pf} The first statement follows directly from Proposition
\ref{ppn:cara_conv_loewner}. For the second statement, fix $\de>0$,
and
note that if we replace $g^j,g$ by the Loewner\vadjust{\goodbreak} chains $h^j,h$
corresponding to the driving functions $W_{x,t\wedge T^j}, W_{x,t\wedge
T}$, respectively, then $h_t^j,h_t$ are defined for all $t\ge0$ (with
$h_t=g_t$ for $t\le T$ and $h_t^j=g_t^j$ for $t\le T^j$). By
Proposition \ref{ppn:cara_conv_loewner}, for all $\ep>0$ and for all $t
< \infty$, we will have $\|h^j-h\|_\infty<\de$ on $[0,t]\times
\overline{\HH}{}^{(\ep)}_{x,t}$ for $j$ sufficiently large. Also, by uniform
continuity of $h$ on $[0,t]\times\overline{\HH}{}^{(\ep)}_{x,t}$,
we will
have $|h_{T^j}(z)-h_T(z)|<\de$ for all $z\in\overline{\HH}{}^{(\ep
)}_{x,T\vee
T^j}$ for $j$ sufficiently large. Therefore
\[
|g_{T^j}^j(z)-g_T(z)|
= |h_{T^j}^j(z)-h_T(z)|
\le|h_{T^j}^j(z) - h_{T^j}(z)| + |h_{T^j}(z) - h_T(z)|
< 2\de
\]
for all $z\in\overline{\HH}{}^{(\ep)}_{x,T\vee T^j}$ and $j$ sufficiently
large. The result\vspace*{-1pt} follows by noting that for any $\ep>0$, we can find
$\ep'>0$ small enough so that $\overline{\HH}{}^{(\ep
)}_{x,T}\subseteq
\overline{\HH}{}^{(\ep')}_{x,s}$ for $s$ sufficiently close to
$T$.\footnote
{This can be checked directly, for example, by similar methods as are
used to prove Proposition~\ref{ppn:cara_conv_loewner}.}
\end{pf}
\begin{rmk} \label{rmk:shifted_curves}
Given a Loewner chain $g\!=\!(g_t)_{0\le t\le T}$ corresponding to $\gam
\!\in\!\allowbreak
\initpathsR$, for any $s<T$ we can also consider the maps $g^{(s)} =
(g_t^{(s)})_{0\le t\le T-s}$ which satisfy $g_{s+t} = g_t^{(s)}\circ
g_s$. These correspond to the curve \mbox{$\gam^{(s)}(t)\defeq g_s(\gam
(s+t))$} defined for $0\le t\le T-s$, and it is easily seen that
$g^{(s)}$ is simply a Loewner chain with driving function
$W_t^{(s)}\defeq W_{s+t}$. Thus, if we have $d_{x}^{R}(\gam^j,\gam
)\to
0$ as in the corollary above, then also $d_{x}^{R}(\gam^{j,(s)},
\gam
^{(s)})\to0$ for any $s<T$.\footnote{This is a slight abuse of notation
since the curves $\gam^{j,(s)}$ and $\gam^{(s)}$ do not necessarily
start at the same point; however $\gam^{j,(s)}(0)$ clearly converges to
$\gam^{(s)}(0)$.}
\end{rmk}

The following corollary will be useful for determining the uniform
limit of a sequence of curves from their Carath\'eodory convergence.
\begin{cor} \label{cor:filling_hausdorff_conv}
Let $x\in\Psi$, and let $\gam^j,\gam\in\initpathsR$ with
$d_{x}^{R}(\gam^j,\gam)\to0$. Then:
\begin{longlist}[(a)]
\item[(a)] For any $\ep>0$, $\overline{\HH}{}^{(\ep)}_x(\gam)$ is a subset
of $\HH
_x(\gam^j)$ for sufficiently large $j$.

\item[(b)] If $U$ is a connected open subset of $\HH$ with $x\in\overline
{U}$, and $U\subseteq\HH_x(\gam^j)$ for large~$j$, then $U\subseteq
\HH
_x(\gam)$.
\item[(c)] If $K_\dagger$ is any $\cdisthaus$ subsequential limit of the
$K_x(\gam^j)$, and $\HH_\dagger$ is the unique component of $\HH
\setminus K_\dagger$ whose closure contains $x$, then $\HH_\dagger
=\HH
_x(\gam)$.
\end{longlist}
\end{cor}
\begin{pf}
Let $(g_{x,t})$ and $(g_{x,t}^j)$ denote the Loewner chains
corresponding to $\gam^j$ and $\gam$, respectively. Throughout this
proof we will use the notation $g=g_{x,T}$ and $g^j=g_{x,T}^j$, where
$T=\capacity_x\gam$ and $T^j=\capacity_x\gam^j$.

(a) By Corollary \ref{cor:cara_from_driving}, $g^j$ must be defined on
$\overline{\HH}{}^{(\ep)}_x(\gam)$ for sufficiently large~$j$.

(b) Suppose for sake of contradiction that $U\cap K_x(\gam)\ne
\varnothing$. By the conditions on~$U$, for any $k>0$ we can
find\vadjust{\goodbreak}
$z\in
\HH_x(\gam)$ such that for some \mbox{$\de>0$}, $\cdist(z,K_x(\gam))<\de
$ and
$B_{k\de}(z)\subseteq U$. Then $z\in\HH_x(\gam^j)$ for large $j$,
and by Carath\'eodory convergence, the conformal radius of $\HH_x(\gam
^j)$ with respect to~$z$ converges to the conformal radius of $\HH
_x(\gam)$ with respect to $z$. But by the Koebe distortion theorem, the
inradius of a domain with respect to an interior point is within a
constant factor of its conformal radius: by choosing~$k$ large enough
we can guarantee that $K_x(\gam^j)$ will intersect $B_{k\de}(z)$ for
large~$j$, which gives the desired contradiction.

(c) By (a) it is clear that $\HH_\dagger\supseteq\HH_x(\gam)$.
Conversely, if $U$ is a connected open subset of $\HH_\dagger$ with
$x\in\overline{U}$ and $\overline{U}\subset\HH_\dagger$, then
$\overline{U}\subset\HH\setminus K_x(\gam^j) = \HH_x(\gam^j)$ for
large~$j$,
and so
by (b) we have $U\subseteq\HH_x(\gam)$. Since $\HH_\dagger$ is a union
of such $\overline{U}$ we find $\HH_\dagger\subseteq\HH_x(\gam)$, hence
they are equal.
\end{pf}

\subsection{Hitting probabilities of Brownian motion}
\label{ssec:hitting_brownian}

Informally, the above co\-rollary says that $d_{x}^{R}$ convergence
gives convergence of the ``shape'' of the fillings~$K_x(\gam^j)$. To
identify the location of the ``tip'' $\gam(T)$ on $K_x(\gam)$, we next
consider hitting probabilities of Brownian motion (i.e., harmonic
measure) for segments of the (conformal) boundary of $\HH_x(\gam)$.

For $\gam\in\initpathsR$ with $T=\capacity_x\gam$, if $x$ is not
swallowed by $\gam$ then we define the \textit{left boundary} (with
respect to $x$) of $\gam$ to be the maximal (closed) clockwise segment
of the conformal boundary of $\HH_x(\gam)$ which begins at $\gam(T)$
and whose intersection with $\R$ has empty interior; we define the
right boundary symmetrically. If $x$ is swallowed by $\gam$ at some
time $\tau_x\le T$, the left boundary is defined to be the set of
points (more precisely, prime ends) on the conformal boundary of $\HH
_x(\gam)$ which lie on the left boundary of $\gam[0,t]$ for any
$t<\tau
_x$. We let $\al_x^z(\gam)$ denote the probability that a Brownian
motion, started at $z\in\HH_x(\gam)$ and stopped upon hitting $\gam
\cup
\R$, will hit the left boundary of~$\gam$.

\begin{ppn} \label{ppn:hit_left_bd}
Fix $x\in\Psi$, and let $\gam^j,\gam\in\initpathsR$ with
$d_{x}^{R}(\gam^j,\gam)\to0$. Then $\al_x^z(\gam^j)\to\al
_x^z(\gam)$.
\end{ppn}

We begin by proving an easier result. First suppose $x$ is not
swallowed by $\gam$. Fix a ``reference point'' $P\in\R$ with $P<\inf
(\gam\cap\R)$, and consider the event that the Brownian motion started
at $z\in\HH_x(\gam)$ will hit either the left boundary of $\gam$ or the
segment $[P,\inf(\gam\cap\R)]$. If this occurs we say that the Brownian
motion hits to the left of the tip with respect to $P$, and we denote
the probability of this event by $\al_x^z(\gam,P)$.
\begin{lem} \label{lem:hit_left_ref_pt}
Fix $x\in\Psi$, and let $\gam^j,\gam\in\initpathsR$ with
$d_{x}^{R}(\gam^j,\gam)\to0$ such that~$x$ is not swallowed by $\gam$.
Then, for any reference point $P$ as above,
\[
\al_x^z(\gam^j,P)\to\al_x^z(\gam,P).
\]
\end{lem}

\begin{pf}
By the conformal invariance of Brownian motion we consider the
problem
in $\D$, with $x=0$, $z\in\D$ and $P\in\pd\D$. Let $g_t,W_t$
denote the
(radial) Loewner\vadjust{\goodbreak} chain and driving function corresponding to $\gam$,
and similarly~$g^j_t,W^j_t$. We write $g=g_T$ and $g^j=g^j_{T_j}$ for
the terminal Loewner maps. By the conformal invariance of Brownian
motion, $\al_x^z(\gam,P)$ is exactly the probability that a Brownian
motion started at $g(z)$ and stopped upon hitting~$\pd\D$ will land on
the arc going counterclockwise from $g(P)$ to $W_T$. Likewise $\al
_x^z(\gam^j,P)$ is the probability that a Brownian motion started at
$g^j(z)$ and stopped upon hitting $\pd\D$ will land on the arc going
counterclockwise from $g^j(P)$ to $W_{T^j}^j$. But $W^j_{T^j}\to W_T$
by assumption, and $g^j\to g$ pointwise on $\overline{\D}\setminus
D_T(\gam)$ by Carath\'eodory convergence, so the result
follows.\vspace*{-3pt}
\end{pf}

\begin{pf*}{Proof of Proposition \ref{ppn:hit_left_bd}}
First suppose $x$ is not swallowed by $\gam$. Write $z_0=\inf(\gam
\cap\R
)$ and $z^j_0=\inf(\gam^j\cap\R)$. We can choose $P<z_0$ to make the
difference $\al^z_x(\gam,P)-\al^z_x(\gam)$ arbitrarily small. On the
other hand
$\al^z_x(\gam^j,P)-\al^z_x(\gam^j)$
is the probability that a Brownian motion started at $z$ and stopping
upon hitting $\gam^j\cup\R$ will land on the segment $[P,z_0^j]$. By
Corollary \ref{cor:filling_hausdorff_conv} we must have $\liminf
z_0^j\ge P$. If $\limsup z_0^j\le P$ then clearly $\al^z_x(\gam
^j,P)-\al
^z_x(\gam^j)\to0$, so the result follows from Lemma
\ref{lem:hit_left_ref_pt}. Therefore suppose $\limsup z_0^j>P$, so that the
curves $\gam^j$ must come close to $z_0$ without touching. Then the
hitting probability of $[z_0,z_0^j]$ must tend to zero (e.g., using the
Beurling estimate), so Lemma \ref{lem:hit_left_ref_pt} again gives the result.

It remains to check the case of when $x$ is swallowed by $\gam$, that
is, $\tau_x(\gam)\le T$. We again work in the unit disc, with $x=0$ and
$z\in\D$. Suppose $c=\al_0(\gam)$ and $\tilde c=\lim_j\al_0(\gam^j)$
with $|c-\tilde c|>\ep$. We have $\lim_{t\incto\tau_0}\al_0(\gam
[0,t])=c$, so for any $\de>0$ we can choose $t<\tau_0$ such that
$\tau
_0-t<\de$ and $|\al_0(\gam[0,t'])-c|<\de$ for all $t'\ge t$; by the
above result we also have for large $j$ that $\al_0(\gam^j[0,t])$ is
within $\de$ of $c$ but $\al_0(\gam^j[0,\tau_0])$ is within $\de$ of
$\tilde c$. Recalling Remark \ref{rmk:shifted_curves}, we now consider
the systems under the maps $g^j_t$: the curves $\gam^{j,(t)}(s)\defeq
g^j_t(\gam(t+s))$ must travel in such a way that $\al_0(\gam
^{j,(t)}[0,s])$ changes by more than $\ep-2\de$ within (capacity) time
$\tau_0-t<\de$. Taking $\de\to0$ we see that this must contradict the
hypothesis that $\gam$ is the initial segment of a continuously driven
curve.\vspace*{-3pt}~%
\end{pf*}

Given $\gam\in\initpathsR$, for any closed subset $S$ of $\gam\cup
\R$
we will let $p^z_x(S;\gam)$ denote the probability that Brownian motion
started at $z$ and stopped upon hitting $\gam\cup\R$ will be stopped on
$S$ (regardless of whether it stops on the left or right boundary of
$\gam$).\vspace*{-3pt}
\begin{cor} \label{cor:hitting_brownian}
Fix $x\in\Psi$, and let $\gam^j,\gam\in\initpathsR$ with
$d_{x}^{R}(\gam^j,\gam)\to0$. Let $T=\capacity_x\gam$ and
$T^j=\capacity
_x\gam^j$. Then, for any $t<T$,
\[
\lim_{j\to\infty} p_x^z(\gam^j[t,T^j]; \gam^j) = p_x^z(\gam
[t,T];\gam).
\]
It follows that for $s<t<T$, $p_x^z(\gam^j[s,t]; \gam^j) \to
p_x^z(\gam
[s,t];\gam)$.
\end{cor}
\begin{pf}
The first claim follows from Remark \ref{rmk:shifted_curves} by
applying Proposition \ref{ppn:hit_left_bd} to the curves $\gam
^{j,(t)}[0,T^j-t]$ and $\gam^{(t)}[0,T-t]$. The second claim is an
immediate consequence,\vadjust{\goodbreak} since $p^z_x(\gam[s,t];\gam)=p^z_x(\gam
[s,T];\gam
)-p^z_x(\gam[t,T];\gam)$.~%
\end{pf}

For our purposes it will suffice to consider only $p_x^x(\cdot;\cdot)$,
which we will denote from now on by $p_x(\cdot;\cdot)$. [Note, however,
that when $\Psi$ is replaced by~$\Q$ or $\{\pm1\}$ it will still be
useful to consider $p_x^z(\cdot;\cdot)$ for general $z\in\HH_x(\gam)$.]

\section{Convergence of deterministic curves}
\label{sec:unif_conv}

In this section we will prove Theorem \ref{thm:unif_conv}.

\subsection{Hausdorff convergence and compatibility}
\label{ssec:hausdorff}

\begin{lem} \label{lem:haus_conv}
Let $\gam^j,\gam\in\initpathsR$ with $d_{x}^{R}(\gam^j,\gam
)\to0$
for all $x\in\Psi$. Then $\cdisthaus(\gam^j,\gam)\to0$.
\end{lem}
\begin{pf}
For any $\ep>0$ it is possible to choose finitely many points
$x_1,\ldots,x_n\in\Psi$ such that $Q\defeq\bigcup_{i=1}^n
\overline{\HH}{}
^{(\ep)}_x(\gam)$ contains every point $z\in\overline{\HH}$ with
$\cdist
(z,\gam)\ge\ep$. Applying Corollary \ref{cor:filling_hausdorff_conv}(a)
to each component separately shows
that $\gam^j\cap Q=\varnothing$ for sufficiently large $j$, hence
$\gam
^j$ is contained in an $\ep$-neighborhood of $\gam$ for sufficiently
large $j$.

For the other direction, let $y\in\gam$, and let $U=B_{\ep}(y)$.
Then $d_{x}^{R}(\gam^j,\gam)\to0$ for some $x\in U\cap\Psi$,
and so
by Corollary \ref{cor:filling_hausdorff_conv}(b) it must be that $U$
intersects~$\gam^j$ for large $j$. Since $\gam$ is compact, it follows
that it will be contained in an $\ep$-neighborhood of $\gam^j$ for
large $j$, which concludes the proof.
\end{pf}

The next lemma tells us that even if we are not given a single curve to
which the $\gam^j$ converge in all the $d_{x}^{R}$ metrics, we can
\textit{almost} construct it from knowing the $d_{x}^{R}$ limits for
each $x\in\Psi$:
\begin{lem} \label{lem:compatibility}
Let $(\gam^j)$ be a sequence in $\initpathsR$, and suppose that for
each $x\in\Psi$ there exists $\gam^x\in\pathspaceR$ with
$d_{x}^{R}(\gam^j,\gam^x)\to0$. Then there exists a unique half-open path
$\eta\dvtx[0,1)\to\overline{\HH}$ such that each $\hat\gam^x\defeq
\gam
^x[0,\tau
_x(\gam^x)]$ is an initial segment of $\eta$ and $\eta=\bigcup
_x\hat\gam
^x$ (up to the inclusion of endpoints).
\end{lem}
\begin{pf}
Let $x_1,x_2$ be two distinct points in $\Psi$. For $i=1,2$ we let
$\gam
^{x_i}$ be parametrized by $\capacity_1$, and set
\[
\tau^i_j=\inf\{t\ge0\dvtx x_j\notin\HH_t(\gam^{x_i})\}.
\]
Set $\si^i=\tau^i_1\wedge\tau^i_2$; this is the first time $t$ such
that either the $x_i$ lie in different components of $\HH\setminus
\gam
^{x_i}[0,t]$, or that they both no longer lie in $\HH_t(\gam^{x_i})$,
the unique component of $\HH\setminus\gam^{x_i}[0,t]$ whose closure
contains $1$. Then set $\si=\si^1\wedge\si^2$; we claim that the
$\gam
^{x_i}$ must agree up to time $\si$. To see this, let $t<\si$: by
definition of $\si$ we must have $\HH_t(\gam^{x_i}) = \HH
_{x_1,t}(\gam
^{x_i}) = \HH_{x_2,t}(\gam^{x_i})$ for $i=1,2$. Let $U$ be a connected
open subset of $\HH$ with $\overline{U}\subset\HH_{x_1,t}(\gam^{x_1})$
and $x_1,x_2\in U$. Then by Corollary
\ref{cor:filling_hausdorff_conv}(a) we have for large $j$ that $U\subset
\HH
_{x_1,t}(\gam^j)$ and $U\subset\HH_{x_2,t}(\gam^j)$, hence $\HH
_{x_1,t}(\gam^j)=\HH_{x_2,t}(\gam^j)$. By Corollary
\ref{cor:filling_hausdorff_conv}(b) we have $U\subseteq\HH_{x_2,t}(\gam
^{x_2})$, and $\HH_{x_1,t}(\gam^{x_1})$ is a union\vadjust{\goodbreak} of sets of the form
$U$ which proves $\HH_{x_1,t}(\gam^{x_1})\subseteq\HH_{x_2,t}(\gam
^{x_2})$, and so by symmetry they are equal. Since the curves are
uniquely determined by their filling processes, the $\gam^{x_i}$ must
agree up to time $\si=\si^1=\si^2$.

It follows that one of the $\hat\gam^{x_i}$ is an initial segment of
the other: $\si=\tau^1_j=\tau^2_j$ for some $j$, and then the curve
$\hat\gam^{x_j}$ must end at time $\si$. Therefore we can let~$\eta
$ be
the union of all $\hat\gam^x$ for $x \in\Psi$; if there is one
$\hat
\gam^x$ of which all other~$\hat\gam^{x'}$ are initial segments, then
$\eta=\hat\gam^x$ is a curve going from $-1$ to $1$. If not, we view
$\eta$ as a half-open path that does not contain its terminal
endpoint.~%
\end{pf}

If $\eta$ is a half-open curve as constructed in Lemma
\ref{lem:compatibility}, we will write $d_{x}^{R}(\gam^j,\eta)\to0$ if
$d_{x}^{R}(\gam^j,\eta')\to0$ for some (all) $\eta'\in
\pathspaceR$
with $\eta'[0,\tau_x(\eta')]=\hat\eta^x$. The following is an example
showing that this $\eta$ need not extend to a closed continuous curve
which contains its terminal endpoint:
\begin{ex} \label{ex:bidir.half.open}
We consider curves traveling between $0$ and $\infty$ within the
closure of the infinite half-strip $D=\{|{\real z}|<1\}\cap\HH$, with the
countable dense subset $\Psi=(\Q\cap D) \setminus(i\R)$. We will adapt
the curve of Example \ref{ex:fwd_insufficient} as follows: for $n\in
\N
$, let $z_n = (-1)^n + in$ as before, but now let $w_n=i(1-2^{-n})$. We
will let $\eta_k$ denote the closed curve which is a linear
interpolation of the points
\[
0, z_1, w_1, \ldots, z_{k-1}, w_{k-1}, z_k.
\]
See Figure \ref{fig:bidir.half.open}(a) and (b). If we let $\eta$ denote the union over all
%
%
\begin{figure}

\includegraphics{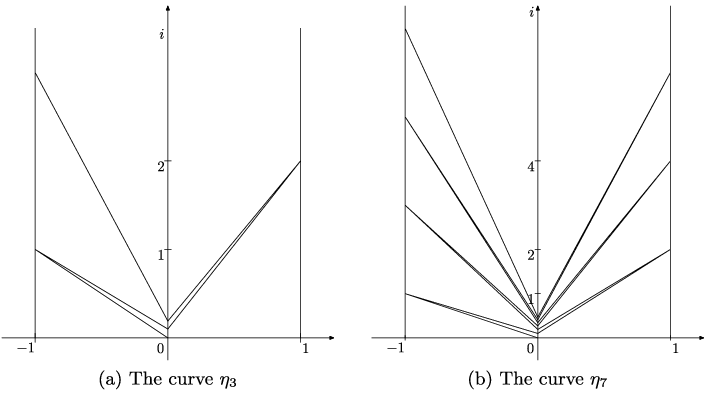}

\caption{Example \protect\ref{ex:bidir.half.open}: $\eta$ does not
extend continuously to closed curve.}
\label{fig:bidir.half.open}
\end{figure}
$\eta_k$, then $\eta$ travels below the line $\{\imag z=1\}$ infinitely
many times, and so it does not extend to a continuous closed curve from
$0$ to $\infty$. Nevertheless it is easy to find a sequence $(\gam^j)$
in $\simplepaths$ such that $d_{x}^{R}(\gam^j,\eta)\to0$ for all
$x\in
\Psi$.
\end{ex}

In the next section we will describe how to use \textit{bidirectional}
driving convergence to obtain the desired continuous extension.

\subsection{Uniform convergence from bidirectional driving convergence}

We begin by proving some useful consequences of the time-separated
assumption. From now on we assume that $(\gam^j)$ is a sequence in
$\simplepathsR$. Since these curves extend continuously to their
endpoint, if we use the $\capacity_1$ parametrization we will write
$\gam^j[t,\infty]$ for the closure of $\gam^j[t,\infty)$.
\begin{defn}
Let $\eta$ be a half-open curve such as constructed by
Lem\-ma~\ref{lem:compatibility}, parametrized by $\capacity_1$. We say that a time
$t_0$ is \textit{nondouble} if $y_0\defeq\eta(t_0)$ is a nondouble point
of $\eta$. We say that $t_0$ is \textit{strongly nondouble} if in
addition $y_0$ does not lie in the closure of $\eta[t,\infty)$ for any
$t>t_0$. We make the symmetric definitions for $\xi$ parametrized by
$\capacity_{-1}$.
\end{defn}
\begin{lem} \label{lem:time_sep_dense_non_double_times}
Let $(\gam^j)$ be a sequence in $\simplepaths$, and suppose that for
each $x\in\Psi$ there exists $\gam^x\in\tspathsR$ with
$d_{x}^{R}(\gam
^j,\gam^x)\to0$. Then the half-open curve $\eta$ constructed by
Lemma \ref{lem:compatibility} has a dense collection of nondouble
times (under the $\capacity_1$ parametrization).
\end{lem}
\begin{pf}
For $0<t_1<t_2<t_3$, we say that $z$ is a $(t_1,t_2,t_3)$-double point
if $\eta(s)=\eta(s')=z$ for some $s\in[0,t_1]$, $s'\in[t_2,t_3]$. Let
$\mathcal D_{t_1,t_2,t_3}$ denote the set of times mapping to
$(t_1,t_2,t_3)$-double points. Then $\mathcal D_{t_1,t_2,t_3}$ is a
closed subset of $\R_{\ge0}$, and since $\eta[0,t_3]$ is
time-separated, it must be that $\mathcal D_{t_1,t_2,t_3}$ has dense
complement in $\R_{\ge0}$: if $\mathcal D_{t_1,t_2,t_3}$ contained a
nontrivial time interval, the interval would map to a nontrivial
connected component of $\eta[0,t_1]\cap\eta[t_2,t_3]$ since $\eta$ is
assumed to be continuously driven (hence not locally constant).

The set $\mathcal D$ of all times mapping to double points can be
expressed as the union of $\mathcal D_{t_1,t_2,t_3}$ over all rational
triples $0<t_1<t_2<t_3$. The countable intersection of open dense sets
is dense by the Baire category theorem, so $\mathcal D$ has dense
complement as desired.
\end{pf}

We now assume the notation and hypotheses of Theorem
\ref{thm:unif_conv}, and let $\eta\defeq\bigcup_x\hat\eta^x$ and $\xi
\defeq
\bigcup_x\hat\xi^x$ be the half-open curves given by Lemma
\ref{lem:compatibility}. At this point we have not yet shown that either
curve extends continuously to its terminal endpoint or that one curve
is the time reversal of the other. However, we know that if there
exists an $x \in\Psi$ that is not swallowed by $\eta$ before its
terminal time, then $\eta= \hat\eta^x$ and hence $\eta$ extends
continuously to its endpoint. We also know that $\HH_x(\eta) =
\HH_x(\eta^x) = \HH_x(\xi^x) = \HH_x(\xi)$ for every $x \in\Psi
$, since
these sets are components of the complement of the Hausdorff limit of
the $\gam^j$ (see Lemma \ref{lem:haus_conv}). Thus, as sets, both
$\eta
$ and $\xi$ contain
\[
\bigcup_{x \in\Psi}\overline{\partial\HH_{x,\infty}(\eta
)\setminus\R},
\]
and both are contained in the closure of this union.\vadjust{\goodbreak}
\begin{lem} \label{lem:strongly_non_double}
Let $\eta$ and $\xi$ be defined as above, parametrized by $\capacity_1$
and $\capacity_{-1}$, respectively. Under this parametrization, each
curve has a dense collection of strongly nondouble times.
\end{lem}
\begin{pf}
Suppose for the sake of contradiction that there is a time interval
$[t_1,t_2]$ (with $t_1<t_2$) in which every time fails to be a strongly
nondouble time for $\eta$. By Lemma
\ref{lem:time_sep_dense_non_double_times}, this time interval does contain
a dense set of nondouble times $t_0$, so it must be that each
corresponding $y_0\defeq\eta(t_0)$ arises as the subsequential limit of
$\eta(t)$ for large $t$. Therefore $S\defeq\eta[t_1,t_2]$ must lie in
the closure of $\eta[t,\infty)$ for any $t>t_2$. We claim that for such
$t$, any subsequential $\cdisthaus$ limit of the sets $\gam
^j[t,\infty
]$ must contain $S$. Indeed, by Lemma~\ref{lem:haus_conv} the limit
must contain $\eta[t,\infty)\setminus\eta[0,t]$. Since $\eta$ is
continuously driven, $\eta[t,\infty) \setminus\eta[0,t]$ is dense in
$\eta[t,\infty)$. Since any $\cdisthaus$ limit is closed, this proves
the claim.

It is clear that we can assume that $t_2$ is a nondouble time. We
therefore consider the following two cases:

\textit{Case} 1. Suppose that after time $t_2$, $\eta$ first hits $(\eta
[0,t_2]\cup\R)\setminus\{\eta(t_2)\}$ at some time $t_3 > t_2$. Let
$x\in\Psi$ be such that $x$ is swallowed at this time, and such that
some nontrivial connected subset $S'$ of $\pd S$ is contained in the
boundary of $U_x\defeq\HH_x(\eta)$. (That such $x$ exists is easy to
see from the definition of $t_3$, for example, by a simple compactness
argument.) By re-labeling we now suppose that the times $t_i$ ($1\le
i\le3$) are all with respect to the $\capacity_x$ parametrization.

Now, in the reverse direction, let $\acute t_i^j$ ($1\le i\le3$) be
defined by $\gam^j(t_i)=\gam^{j-}(\acute t_i^j)$ (again, with respect
to $\capacity_x$). By passing to a subsequence we may suppose that
$\acute t_i^j$ converges to some $\acute t_i$ for each $i$. By
hypothesis, $\gam^{j-}[0,\acute t_3^j]$ converges to $\xi[0,\acute
t_3]$ with respect to $d_{x}^{L}$ for all $x\in\Psi$, and by
Lemma \ref{lem:haus_conv} it converges in $\cdisthaus$ as well, so by
the first claim above $\xi[0,\acute t_3]\supseteq S$. Since $\xi$ is
time-separated, during the time interval $[\acute t_3,\acute t_1]$ it
can only hit a (closed) totally disconnected subset of $S$. Therefore,
we can find a point $y$ in the interior of $S'$ (in the subspace
topology on $S'$) such that a small neighborhood $V$ of $y$ is not hit
by $\xi[\acute t_3,\acute t_1]$. We can then choose $z\in U_x$ close
enough to $y$ such that with probability at least $1-\ep$ (for small
$\ep>0$), a~Brownian motion started at $z$ and stopped upon hitting
$\pd U_x$ will stop on $V\cap S'$, so that $p^z_x(\xi[\acute t_2,
\acute t_1];\xi)\le\ep$. But on the other hand $S'\subseteq\eta
[t_1,t_2]$, so $p^z_x(\eta[t_1,t_2];\eta)\ge1-\ep$. The contradiction
follows from noting that by Corollary \ref{cor:hitting_brownian},
\begin{eqnarray*}
p^z_x(\eta[t_1,t_2];\eta)
&=& \lim_{j\to\infty} p^z_x(\gam^j[t_1,t_2];\gam^j)
= \lim_{j\to\infty} p^z_x(\gam^{j-}[\acute t_2^j,\acute
t_1^j];\gam^{j-})\\
&=& p^z_x(\xi[\acute t_2,\acute t_1];\xi).
\end{eqnarray*}

\textit{Case} 2. Now suppose $\eta$ never hits $(\eta[0,t_2]\cup\R
)\setminus\{\eta(t_2)\}$ after time $t_2$, and let $z$ be any point of
$\Psi$ that is swallowed by $\eta$ after time $t_2$, say at time $t_3$.
Note that $\eta(t_2)$ forms a cut point of $\eta[0,t_3]$. Now, $\eta$
must swallow every point of $\Psi$ eventually: if some $x\in\Psi$ does
not get swallowed then we would have $\eta^x=\eta$, but by hypothesis
the $\eta^x$ lie in $\tspaths$, and hence extend continuously to their
endpoints, while~$\eta$ does not. Consider those points \mbox{$x\in\Psi$}
which lie in a neighborhood of the cut point $\eta(t_2)$ and which have
not been swallowed by time $t_2$: since all of these points must
eventually be swallowed (but they cannot all be swallowed at once since
$\eta$ never hits $(\eta[0,t_2]\cup\R)\setminus\{\eta(t_2)\}$), we
see that the closure of $\eta[t_3,\infty)$ must surround~$z$, and thus
the~$\cdisthaus$ limits of both $\gamma^j[t_2,t_3]$ and $\gamma
^j[t_3,\infty]$, which we denote~$B$ and~$B'$, must surround~$z$. $B$~and $B'$ are connected sets, and neither is contained in the other
since~$\eta$ and~$\xi$ are continuously driven, but by construction $B$
will be ``nested'' inside~$B'$. This contradicts the assumption that
the $\gam^{j-}$ converge with respect to~$d_{z}^{R}$ to a
continuously driven curve.
\end{pf}

Note that it follows immediately that there is a dense collection of
strongly nondouble times mapping to points not in $\R$, since for
continuously driven curves the set of times mapping into $\R$ is closed
and totally disconnected.
\begin{lem} \label{lem:bidirectional}
Assume the notation and hypotheses of Theorem \ref{thm:unif_conv}, and
let $\eta\defeq\bigcup_x\hat\eta^x$ and $\xi\defeq\bigcup_x\hat
\xi^x$
be the half-open curves given by Lemma \ref{lem:compatibility},
parametrized by $\capacity_1$ and $\capacity_{-1}$, respectively. Let
$z_0=\eta(t_0)$ be a strongly nondouble point of $\eta$ with
$z_0\notin
\R$. Fix $t_*>t_0$, and let $\eta^j_*$ denote the curve $\gam^j$
stopped at time~$t_*$. Let $\bar\eta^j_*$ denote the remaining curve
$\gamma^j \setminus\eta^j_*$. Then, if $X$ is any (subsequential)
$\cdisthaus$ limit of the $\bar\eta^j_*$, we will have $z_0\notin X$.
\end{lem}
\begin{pf}
Since $z_0$ is a strongly nondouble point, we can find $\ep>0$ small
enough so that $\eta$ does not enter $\overline{B_{\ep}(z_0)}$ after
time $t_*$. We further require \mbox{$B_{\ep}(z_0)\subset\HH$}. By the
Beurling estimate, for any $\de>0$ we may choose $0<\ep'<\ep$ small
enough so that for any curve $P$ crossing the annulus $\{\varepsilon
'<|z-z_0|<\varepsilon\}$, a Brownian motion started inside $B_{\ep
'}(z_0)$ has probability less than $\de$ of exiting $B_{\ep}(z_0)$
without hitting~$P$. Thus for $x\in B_{\ep'}(z_0)\cap\Psi$ we will
have $p_x(\eta[t_*,\infty);\eta)<\de$ and $p_x(\R;\eta)<\de$,
and so by
Corollary \ref{cor:hitting_brownian}
%
%
\begin{equation} \label{eq:bidir_fwd}
\lim_{j\to\infty} p_x(\bar\eta^j_*;\gam^j)<\de\quad\mbox{and}\quad
\lim_{j\to\infty} p_x(\R;\gam^j)<\de.
\end{equation}

We now consider the reverse direction. Let $0<\ep''<\ep'$ and
\[
s_1\defeq\inf\{t\ge0\dvtx\xi(t)\in B_{\ep''}(z_0)\},
\]
and note that since $\gam^{j-}[0,s_1]$ has $\cdisthaus$ limit $\xi
[0,s_1]$ not containing $z_0$, a fortiori we have $\gam
^{j-}[0,s_1]\subset\bar\eta^j_*$ for large $j$. Now, making use of
Lemma \ref{lem:strongly_non_double}, let $s_0<s_1$ be a strongly
nondouble time of $\xi$ such that $\acute z_0=\xi(s_0)$ lies in $\{
\ep''<|z-z_0|<\ep'\}$, and let $\tilde\ep>0$ be small enough so that
$\xi$ does not enter $B_{\tilde\ep}(\acute z_0)$ after ti\-me~$s_1$.
Applying the Beurling estimate again we choose $0<\tilde\ep'<\tilde
\ep$
small enough so that for any curve $P$ crossing $\{\tilde\ep
'<|z-\acute z_0|<\tilde\ep\}$, a Brownian motion started inside
$B_{\tilde\ep'}(\acute z_0)$ has probability less than $\de$ of exiting
$B_{\tilde\ep}(\acute z_0)$ without hitting $P$. Then for $x\in
B_{\tilde\ep}(\acute z_0)\cap\Psi$ we have $p_x(\xi[s_1,\infty
);\xi
)<\de
$, and so by our observation above
%
%
\begin{equation} \label{eq:bidir_back}
\lim_{j\to\infty} p_x(\eta^j_*;\gam^j)
\le\lim_{j\to\infty} p_x(\gam^{j-}[s_1,\infty];\gam^{j-})<\de.
\end{equation}
Combining (\ref{eq:bidir_fwd}) and (\ref{eq:bidir_back}) gives
\[
\lim_{j\to\infty} [
p_x(\eta^j_*;\gam^j)
+ p_x(\bar\eta^j_*;\gam^j)
+ p_x(\R;\gam^j)
] < 3\de
\]
and setting $\de\le1/3$ we obtain a contradiction, since a Brownian
motion started at $x$ and stopped upon hitting $\eta\cup\R$ must
clearly be stopped at one of $\eta^j_*$, $\bar\eta^j_*$ or $\R$.
\end{pf}
\begin{cor} \label{cor:bidirectional}
Assume the notation and hypotheses of Theorem \ref{thm:unif_conv} and
Lemma \ref{lem:bidirectional}. The paths $\eta,\xi$ extend continuously
to their endpoints and $\eta=\xi^-\defeqr\gamma$.
\end{cor}
\begin{pf}
First, notice that if $z_1,z_2$ are strongly nondouble points of $\eta
$ not in $\R$ such that $\eta$ hits $z_1$ before $z_2$, then $\xi$ hits
both these points, and will hit $z_2$ before~$z_1$. Indeed, writing
$t_i=\eta^{-1}(z_i)$, let $t_*\in(t_1,t_2)$: by Lemma
\ref{lem:bidirectional}, any subsequential $\cdisthaus$ limit $X$ of the
curves $\bar\eta_*^j=\gam^j[t_*,\infty]$ will be a closed initial
segment of $\xi$ containing $\eta\setminus\eta[0,t_*]$ (hence the point
$z_2$) but not $z_1$.

Suppose now that $\eta(t)$ has multiple limit points as $t\to\infty$,
that is, that there exist sequences $t_{i,k}$ (for $i=1,2$) such that
$z_i = \lim_k \eta(t_{i,k})$ with $z_1\ne z_2$. We may assume that all
the $t_{i,k}$ map to strongly nondouble points of $\eta$ not in $\R$,
by the continuity of $\eta$ and the density of such times. But then we
claim that for any $\varepsilon>0$ the curve $\xi$ must travel between
$B_{\varepsilon}(z_1)$ and $B_{\varepsilon}(z_2)$ infinitely
many times
before time $s$ for some $s>0$, which contradicts the continuity of
$\xi
$. Therefore $\eta$ extends continuously to its endpoint, and it
follows from the above that $\eta=\xi$ as sets.

It remains to show that $\xi=\eta^-$. We know that there is a dense set
of times $t_k$ which map to strongly nondouble points of $\eta$ not in
$\R$, and that $\xi$ hits these points in reverse order. Under the
$\capacity_{-1}$ parametrization of $\xi$, let $\mathcal T = \{t \dvtx
\xi
(t) = \eta(t_k)\mbox{ for some } k\}$. If $\mathcal T$ is a dense set
of times for $\xi$ then we are done, so suppose that there is an
interval of time $I$ not contained in $\overline{\mathcal T}$. But
$\xi
(\overline{\mathcal T}) = \eta[0,\infty) = \xi[0,\infty)$, and so
$I$ is
an interval of times mapping to double points, which gives the contradiction.
\end{pf}
\begin{pf*}{Proof of Theorem \ref{thm:unif_conv}}
Let $\eta\defeq\bigcup_x\hat\eta^x$ and $\xi\defeq\bigcup_x\hat
\xi^x$
be the half-open curves given by Lemma \ref{lem:compatibility}. Thanks
to Corollary \ref{cor:bidirectional} we finally know that the
(strongly) nondouble points of $\eta$ and $\xi$ are the same, as
$\gam
=\eta=\xi^-$. Let\vadjust{\goodbreak} $z_1=\gam(t_1),z_2=\gam(t_2)$ be nondouble points of
$\gam$ with $t_1<t_2$, and let $\gam[z_1,z_2]$ denote the portion of
$\gam$ between these two hitting points, viewed as a closed set. We
then let $\gamma^j[z_1,z_2]$ denote the portion of $\gamma^j$ between
its nearest approach to $z_1$ and its nearest approach to $z_2$. (If
there is a tie, we choose the earliest one, say.) To show uniform
convergence, it suffices to prove that for any such $z_1$ and $z_2$, we
have $\gamma^j[z_1,z_2]$ converging along subsequences in the
$\cdisthaus$ metric to subsets of $\gamma[z_1,z_2]$.

Let $X$ denote any subsequential $\cdisthaus$ limit of the $\gam
^j[z_1,z_2]$, and let $z_0=\gam(t_0)\notin\R$ be a nondouble point of
$\gam$ with $t_0\notin[t_1,t_2]$. We claim that $z_0\notin X$: without
loss of generality we assume $t_0>t_2$; then, by Lemma
\ref{lem:bidirectional} applied to the reverse path $\gam^-$, it suffices
to show that for some $t\in(t_2,t_0)$, the $\gam^j[z_1,z_2]$ are
stopped before time $t$ for sufficiently large $j$. But if such a $t$
does not exist, it means that for any $t\in(t_2,t_0)$, as $j\to\infty$,
we can find subsequences along which the nearest approach of $\gam^j$
to $z_2$ occurs after time $t$. But since~$z_2$ is in the $\cdisthaus$
limit of $\gam^j[0,t_2]$, this shows that it will be in the
$\cdisthaus
$ limit of $\gam^j[t,\infty]$ as well, which is a violation of
Lemma \ref{lem:bidirectional} applied to the forward path $\gam$.

It follows that $X$ is a connected subset of $\gam$ contained in
\[
C=\gam[0,t_2]\cap\gam[t_1,\infty)
= \gam[t_1,t_2] \cup\bigl( \gam[0,t_1] \cap\gam[t_2,\infty) \bigr)
\cup(\gam\cap\R).
\]
Since $\gam$ is continuously driven and time-separated, $\gam[0,t_1]
\cap\gam[t_2,\infty)$ and $\gam\cap\R$ are both totally disconnected
closed sets, hence their union is as well. Any point of $C$ not
contained in $\gam[t_1,t_2]$ therefore forms a trivial connected
component of $C$, so we must have $X\subseteq\gam[t_1,t_2]$, and the
statement of the theorem follows.
\end{pf*}
\begin{pf*}{Proof of Proposition \ref{ppn:unif_conv_simple}}
This result is essentially a restatement of the remarks following the
list of examples in Section \ref{subsec:main_result}: if $\gam$ is a
simple curve, then $\dcapf(\gam^j,\gam)\to0$ and $\dcapb(\gam
^{j-},\gam
^-)\to0$ together suffice to guarantee~$d_{x}^{R}$ and $d_{x}^{L}$
convergence with respect to all $x$ in a dense subset of $\overline
{\HH}$.~%
\end{pf*}

\subsection{Alternate proof of uniform convergence}

In this section we sketch an alternative proof of Theorem
\ref{thm:unif_conv} that does not use the lemmas of the previous section.
Instead of showing the existence of nondouble and strongly nondouble
times---and considering segments of the path between these times---we
begin by constructing a parametrization of $\eta$ and $\xi$ that
behaves well under time-reversal.

Suppose first that $\eta$ is parametrized by capacity time $t$. For any
\mbox{$x \in\Psi$}, observe that $f_x \defeq g_{x,\infty}^{-1} \circ\vph
^{-1}$ is a conformal map from the unit disc $\D$ to~$\HH_x(\eta)$ that
extends continuously to $\overline{\D}$. Thus, for any $t$,
\[
A_x(t) \equiv A^\eta_x(t) = \{\thet\in[0,1] \dvtx f_x(e^{2\pi i\thet})
\in\eta[0,t]\cap\HH\}
\]
is a closed subset of $[0,1]$. Let $s_x(t) \equiv s^\eta_x(t)$ be the
(Lebesgue) measure of the \textit{interior} of $A_x(t)$, divided by the
measure of $A_x(\infty)$. Using topological arguments, it is not hard
to see that this interior contains at most one interval of $[0,1]$\vadjust{\goodbreak}
(viewed topologically as a circle, by identifying endpoints). Thus~$s_x(t)$ is proportional to the length of this interval. Informally,
$s_x(t)$ represents the portion (in terms of harmonic measure from $x$)
of $\partial\HH_x(\eta) \setminus\R$ that has been traced by time
$t$. Because $\eta$ is continuous, $s_x(t)$ is a continuously
increasing function of $t$.

Now fix a map $a\dvtx\Psi\to(0,\infty)$ with $\sum_{x \in\Psi} a(x) = 1$
and write\vspace*{2pt} $s(t) =\break \sum_{x \in\Psi} a(x) s_x(t)$.\vspace*{2pt} We claim that $s$ is
\textit{strictly} increasing function of $t$. This follows from the
time-separation assumption and arguments in the proof of
Lemma~\ref{lem:strongly_non_double}, which show that an open dense subset of
$\partial(\HH\setminus\eta[0,t])\setminus\R$ will remain on the
boundary of $\partial(\HH\setminus\eta[0,\infty))\setminus\R$.

We therefore take $s \in[0,1)$ to be our new parametrization of $\eta
$, and we can assume that $\xi$ is analogously parametrized by $[0,1)$.
The proof now proceeds with the following observations:
\begin{enumerate}
\item If $t^j$ is any sequence of times then the sets $\gamma^j[0,t^j]$
and $\gam^j[t^j,\infty]$ must converge (subsequentially) in
$\cdisthaus
$ to $\eta[0,s]$ and $\xi[0,1-s]$ for some $s$. Indeed, using Lemma
\ref{lem:haus_conv} we already have Hausdorff convergence to some~$\eta
[0,s]$ and $\xi[0,s']$, and need only check that $s'=1-s$. This
involves checking for each $x$ that the Hausdorff limits of $\gamma
^j[0,t^j]$ and $\gamma^j[t^j,\infty]$ cannot contain overlapping
intervals of $\partial\HH_x(\eta)$, even though the union of these two
limits and $\R$ includes all of $\partial\HH_x(\eta)$. This is done
with the same arguments as those used in the previous section: if the
intervals overlapped, then either $\gamma^j$ or its time reversal would
fail to converge with respect to $x$ to a continuously driven limit.
\item The curve $\eta$ extends continuously to $[0,1]\dvtx\lim_{s \to1}
\cdisthaus(\xi[0,1-s], \{1\}) = 0$, but $\xi[0,1-s]$ is the Hausdorff
limit of $\gamma^j[t_j,1]$ (for some sequence $t_j$), and by the above
this must contain $\eta\setminus\eta[0,s)$, a dense subset of $\eta
[s,1]$. Since $\xi[0,1-s]$ is closed it must contain $\eta[s,1]$ which
gives the claim.
\item The above imply that $\eta(s) = \xi(1-s)$ for all $s \in[0,1]$.
\item For any pair of sequences of times $t^j_1$ and $t^j_2$ such that
$\gamma^j[0,t^j_1]$ tends to $\eta[0,a]$ and $\gamma^j[t^j_2,1]$ tends
to $\eta[b,1]$, we have Hausdorff convergence of $\gamma
^j[t^j_1,t^j_2]$ to a closed subset of $\eta[0,a] \cap\eta[b,1]$,
which by time-separation must be simply $\eta[a,b]$.
\end{enumerate}
The latter item implies convergence in $\dstrong$.

\section{Extension to random curves}
\label{sec:unif_conv_random}

Now we use Proposition \ref{ppn:separable} and
Theorem~\ref{thm:unif_conv} to prove Theorem \ref{thm:unif_conv_random}:
\begin{pf*}{Proof of Theorem \ref{thm:unif_conv_random}}
Each $\gam^j$ can be viewed as a random variable taking values in
\[
\Om_\Psi=\prod_{x\in\Psi} \Gamma_{x}^R \times\Gamma_{x}^L,
\]
where $\Gamma_{x}^R$ is the Polish (complete separable metric)
space defined by the completion of $\pathspaceR/\sim_x$ with respect to
$d_{x}^{R}$, and similarly $\Gamma_{x}^L$.\vadjust{\goodbreak}

Prohorov's criterion (see, e.g., \cite{billingsleyconv}) states that
a family $\Pi$ of probability measures on a complete separable metric
space is relatively compact (in the topology of weak convergence) if
and only if for every $\varepsilon$ there is a compact set $K$ such that
$\mu(K) \ge1-\varepsilon$ for all $\mu\in\Pi$. By hypothesis the marginal
laws of the $\gam^j$ on each $\Gamma_{x}^R$ (or $\Gamma_{x}^L$)
form a relatively compact family, so for each~$\ep$ we can find compact
sets $K_x^R\subset\Gamma_{x}^R$, $K_x^L\subset\Gamma_{x}^L$
such that
\[
\sum_{x\in\Psi} [\PP(\gam^j\notin K_x^R) + \PP(\gam^j\notin
K_x^L)]\le
\ep.
\]
By Tychonoff's theorem, the product $K=\prod_{x\in\Psi} (K_x^R\times
K_x^L)$ is also compact and has probability at least $1-\ep$. Applying
Prohorov's criterion again, we see that the laws of the $\gam^j$ form a
relatively compact family of measures on~$\Om_\Psi$.

Take a subsequence of the $\gam^j$ which converges in law (as $\Om
_\Psi
$-valued random variables) to a random element $\gam\in\Om_\Psi$.
Recall the Skorohod--Dudley theorem \cite{dudleyskorohod}, which
states that random variables on a complete separable metric space
converge in law to a limit if and only if there is a coupling in which
they converge almost surely. Thus we can define the $\gam^j$ of this
subsequence on the same probability space so that $\gam^j\to\tilde
\gam$
a.s. in $\Om_\Psi$. By the hypothesis of the theorem, $\tilde\gam$ has
the marginal law of~$\eta^x$ in each~$\Gamma_{x}^R$, and of~$\xi^x$
in each $\Gamma_{x}^L$, and so we can further couple the sequence
with $\eta^x$ and $\xi^x$ so that $d_{x}^{R}(\gam^j,\eta^x)\to
0$ and
$d_{x}^{L}(\gam^j,\xi^x)\to0$ a.s. for each $x\in\Psi$. Thus,
applying Theorem \ref{thm:unif_conv} we have $\dstrong(\gam^j,\gam
)\to
0$ for some random curve $\gam\in\pathspace$, which depends a priori on
the particular subsequence. However, we have a.s. that for each $x\in
\Psi$, $\eta^x$ is an initial segment of $\gam$ while $\xi^x$ is a
concluding segment. The marginal laws of the $\eta^x,\xi^x$ are
uniquely specified by the hypothesis of the theorem, and by taking $x$
arbitrarily close to the endpoints of $\gam$ we conclude that the law
of $\gam$ as a $\pathspace$-valued random variable is uniquely
specified also.

The above shows that every subsequence of the $\gam^j$ has a further
subsequence that converges in law to $\gam$ with respect to $\dstrong$;
this of course implies that the entire sequence $\gam^j$ converges in
law to $\gam$ with respect to $\dstrong$.
\end{pf*}

\section{Application to $\sle$ curves}
\label{sec:app_sle}

$\sle_\ka$ ($\ka<8$) misses $\Psi$ a.s. Thus, to apply our result to
$\sle$ curves, we need only show that the curves are a.s. time-separated:
\begin{lem} \label{lem:time_sep}
Let $\gam$ be a (random) $\sle_\ka$ curve traveling from $-1$ to $1$.
For $\ka<8$, $\gam\in\tspathsR$ a.s.
\end{lem}
\begin{pf}
For $\ka\le4$ this holds trivially since $\sle_\ka$ is a.s. simple. It
is also not hard to show that when $\kappa\in(4,8)$, the path $\sle
_\kappa$ is almost surely time-separated.\vadjust{\goodbreak} A much stronger set of
results is proved for the so-called $\sle_{\kappa; \kappa- 4, \kappa
-4}$ process in~\cite{dubedat}, Section 3. The set $X$ of cut point
times of an $\sle_{\kappa; \kappa-4, \kappa-4}$ curve $\gam_0$ is shown
to have the same law as the range of a stable subordinator with index
$2-\kappa/4$ (and in particular is totally disconnected)
(\cite{dubedat}, Corollary 13), and the path $\gam_0$ a.s. never revisits a
cut point, so that $\gam$ is injective on $X$. Given the cut point
times, the driving function restricted to each interval of $[0,\infty)
\setminus X$ (modulo additive constant) is independent of the driving
function (modulo additive constant) restricted to the other intervals
(see \cite{dubedat}, Section 3, Lemma 12 and Corollary 13). In other
words, the increments corresponding to the various intervals are
independent of one another. Each increment describes the ``bead''
traced by $\gam_0$ in between the two cut points, and it is easy to see
that each bead has at least a positive probability of having its left
and right boundaries both be nontrivial; thus there will almost
certainly be countably many such beads between each pair of cut points,
and this implies that $\gam_0(X)$ is a.s. totally disconnected, or
equivalently, that the intersection of the left and right boundaries of
$\gamma_0$ is totally disconnected. In between visits to $\R$, the
trace of an $\sle_{\kappa}$ has a~law which is absolutely continuous
with that of $\sle_{\kappa; \kappa- 4, \kappa-4}$~\cite{schrammwilson}.
From this one may deduce that if $\gam$ is an $\sle
_{\kappa}$ and $t$ is any fixed time, then the intersection of the left
and right boundaries $L_t$ and $R_t$ of $K_t$ is also a.s. totally
disconnected.

Now, $\gam[0,t]\cap\gam[t,\infty)$ must lie in $L_t\cup R_t$. Also,
$L_t\setminus R_t$ and $R_t\setminus L_t$ are mapped injectively into
$\R$ by $g_{1,t}$. Since the intersection of $\sle_\ka$ ($\ka<8$) with~$\R$ is totally disconnected a.s. (see, e.g.,
\cite{rohdeschramm}, Theorem 6.4), any connected component of $\gam[0,t]\cap\gam
[t,\infty)$ must lie in $L_t\cap R_t$. But as we saw above this set is
totally disconnected, and so we have a contradiction.
\end{pf}
\begin{pf*}{Proof of Corollary \ref{cor:app_sle}}
Lemma \ref{lem:time_sep} implies that the $d_{x}^{R}$ and $d_{x}^{L}$
limits of the $\gam^j$ are a.s. in $\tspathsR$ and $\tspathsL$,
respectively, so the result follows from Theorem~\ref{thm:unif_conv_random}.
\end{pf*}
\begin{pf*}{Proof of Corollary \ref{cor:app_sle_generalized}}
Lemma \ref{lem:time_sep} implies that the $d_{x}^{R}$ limits of the
$\gam^j$ are a.s. in $\tspathsR$, and by hypothesis the $d_{x}^{L}$
subsequential limits are in~$\tspathsL$, so the result follows from
Theorem \ref{thm:unif_conv_random}.
\end{pf*}

Now that we have proved Corollary \ref{cor:app_sle_generalized}, it is
worth remarking that Schramm and Wilson \cite{schrammwilson} have
given a complete characterization of driving functions for the forward
direction of $\sle$ viewed from different points, which we briefly
describe in our current context: let $\ka\ge0$ and $\rho\in\R$, and
consider the solution of the system
%
%
\begin{equation} \label{eq:sle_kappa_rho}
dW_t = \sqrt{\ka} \,dW_t + i \f{\rho}{2} \biggl( \f
{e^{iW_t}+V_t}{e^{iW_t}-V_t} \biggr) \,dt,\qquad
dV_t = -V_t \f{V_t + e^{i W_t}}{V_t - e^{i W_t}}
\end{equation}
with initial condition $(W_0;V_0)=(w_0;v_0) \in\partial\D$. The radial
Loewner chain obtained from\vadjust{\goodbreak} the driving function $W_t$ is a \textit{radial
$\sle_{\ka;\rho}$ in $\D$} started at $(w_0;v_0)$; $v_0$ is thought of
as a ``force point'' which adds some drift to the usual $\sle_\ka$
driving function. The conformal image of this random curve under $\vph
^{-1}$ is called a \textit{radial $\sle_{\ka;\rho}$ in $\HH$} started at
$(\vph^{-1}(w_0); \vph^{-1}(v_0))$. It was shown in
\cite{schrammwilson} that if $\gam$ is a standard chordal $\sle_\ka$
traveling in the upper half-plane between two boundary points $a,b$,
and $x=x_1 + ix_2$ is any point in $\HH$, then $\psi_x\gam$ is a radial
$\sle_{\ka;\ka-6}$ in $\HH$ started at $(w_0;v_0) = (\psi_x(a);
\psi
_x(b))$, and so the driving function $W_{x,t}$ is given by the solution
to (\ref{eq:sle_kappa_rho}) with $\rho=\ka-6$ and initial condition
$(W_0;V_0)=(\vph\psi_x(a);\vph\psi_x(b))$. (For $\ka=6$, $W_{x,t}$ is
simply a standard Brownian motion.)
\begin{pf*}{Proof of Corollary \ref{cor:app_sle_simple}}
Follows from Theorem \ref{thm:unif_conv} by (a simplified version of)
the proof of Theorem \ref{thm:unif_conv_random}.
\end{pf*}

\section*{Acknowledgments}

The idea of our main result---in the special case of a~simple
limiting curve---emerged during the first author's collaboration with
the late Oded Schramm on the convergence of level lines of the Gaussian
free field to $\sle_4$ and $\sle_{4;\rho}$
\cite{schrammsheffieldgff}.  The general problem was not solved at
that time (\cite{schrammsheffieldgff} employs a GFF-specific
topology-strengthening argument instead), but the seed was planted, and
we continue to benefit from Schramm's insight and encouragement. We
also thank Yuval Peres and the MSR Theory Group for supporting the
visit to Redmond during which this work was partially completed. The
second author thanks Amir Dembo for advice and support. We thank Jason
Miller and Steffen Rohde for very helpful feedback on a draft of this
paper.


%
\printaddresses

\end{document}